\documentclass[11pt]{article}
\usepackage{hyperref}
\usepackage{color}
\usepackage{amsmath,amsthm,amssymb,amsfonts}
\usepackage{enumerate}
\usepackage{graphicx}
\usepackage{bm}
\usepackage{dsfont}
\usepackage{algorithm,algpseudocode} 
\usepackage{soul}

\graphicspath{{./}{figures/}}

\def\R{\mathbb R}
\def\N{\mathbb N}

\newcommand{\M}{\mathbb{M}}
\def\dx{{\Delta x}}
\def\s{\sigma}
\def\t{\tau}
\def\hphi{{\hat \phi}}
\def\bc{{\bar c}}

\numberwithin{equation}{section}

\newcommand{\dis}{\displaystyle}
\newtheorem{theo}{\bf Theorem}[section]

\newtheorem{pro}[theo]{\bf Proposition}

\newtheorem{defi}[theo]{\bf Definition}

\theoremstyle{remark}

\newtheorem{rem}[theo]{\bf Remark}

\def\Um{\mathbf{U}}

\newcommand{\dt}{\,\mathrm{d}t}
\DeclareMathOperator{\dist}{\mathrm{dist}}
\newcommand{\distH}{\dist_\mathrm{H}}

\newcommand{\Man}[1]{\mathbf{M}^{#1}}

\def\a{a}

\newcommand{\BL}{\mathbf{BCL}}
\newcommand{\B}{\mathbf{B}}
\renewcommand{\H}{\mathcal{H}}
\renewcommand{\S}{\mathcal{S}}
\newcommand{\cT}{\mathcal{T}}

\newcommand{\cob}{\overline{\mathop{\rm co}}}

\def\uh{u_{h,\dx}}
\def\xh{x_{h,\dx}}
\def\xd{x_\Delta}
\def\xdj{x_{\Delta,\s}}
\def\ud{u_\Delta}

\def\HT{H_T}
\newcommand{\T}{\mathcal{\Im}}

\renewcommand{\d}{\,\mathrm{d}}

\begin{document}

\title{Approximation of the value function\\for optimal  control problems on stratified domains}

\author{Simone Cacace\thanks{Dipartimento di Matematica e Fisica, Università Roma Tre, Largo S. Leonardo Murialdo 1, 00146 Roma (Italy). Email: {\tt  cacace@mat.uniroma3.it}} \and Fabio Camilli \thanks{Dipartimento di Scienze di Base e Applicate per l'Ingegneria, Universit\`a di Roma ``La Sapienza", via Scarpa 16, 00161 Roma (Italy). Email: {\tt  fabio.camilli@uniroma1.it}} }

\maketitle

\begin{abstract} 
	In optimal control   problems defined on stratified domains,  the dynamics and the running cost may have discontinuities on a finite union of submanifolds of $\R^N$. In \cite{BrYu,BaCh_AEYW}, the corresponding value function is characterized as the unique viscosity solution of a discontinuous Hamilton-Jacobi equation satisfying additional  viscosity conditions  on   the submanifolds. In this paper, we consider a semi-Lagrangian approximation scheme for the previous problem. Relying on a classical stability argument in viscosity solution theory, we prove the convergence of the scheme to the value function. We also present \texttt{HJSD}, a free software we developed for the numerical solution of control problems on stratified domains in two and three dimensions, showing, in various examples, the particular phenomena that can arise with respect to the classical continuous framework.
\end{abstract}



\section{Introduction}
For  discontinuous control problems,  the definition of viscosity solution  is, in general, not sufficient  to guarantee the characterization of the value function as the unique  solution of the corresponding Hamilton-Jacobi equation.  Hence, the goal is both to find sufficient conditions and to appropriately modify the definition of viscosity solution in order to recover the previous property (see for example \cite{BBC1,BBCI,CamSic, GigHam, Sor}).\\
In the seminal paper \cite{BrYu}, Bressan and Hong introduced a class of optimal control problems defined on stratified domains;  namely, the state space $\R^N$ admits a stratification given by a disjoint union of submanifolds of $\R^N$ where the controlled  dynamics and running cost may have discontinuities. Uniqueness of solutions is obtained by adding differential conditions to the notion of viscosity solution in the tangent direction  to the submanifolds of the stratification. Then, the problem was   studied in detail by Barles and Chesseigne \cite{BaCh_AEYW} who provided, under natural regularity and controllability  assumptions    for the data of the problem, a Comparison Principle  and a stability result for   semicontinuous sub and supersolutions of stratified Hamilton-Jacobi equations (see \cite{BaCh_book} for a complete account).\\
In this paper, we study a  semi-Lagrangian scheme for the approximation of stationary Hamilton-Jacobi equations on flat stratified domains. This type of schemes, together   with finite differences ones, constitute a classical method to approximate the viscosity solution of a  Hamilton-Jacobi equation (see \cite{FalFer}). Their convergence to the value function is usually demonstrated through  the semi-relaxed limits technique introduced in \cite{BaSo}, which has, as its fundamental ingredient, the  Comparison Principle.\\
 As usual for semi-Lagrangian schemes, the considered  approximation scheme is based  on the controllistic interpretation of the problem and, in particular, it takes into account the control problems defined on each submanifold. We first prove   existence and uniqueness of the solution to the discrete problem. Then, to prove its convergence, we follow  the classical approach in \cite{BaSo}, and we show that the scheme is consistent in viscosity sense  not only with the Hamilton-Jacobi equation defined on all $\R^N$, but also with the tangential Hamilton-Jacobi equations defined on the submanifolds of the  stratification. In this way, the semi-relaxed limits of the sequence of the discrete solutions  satisfy the additional conditions for stratified viscosity solutions  introduced in \cite{BrYu,BaCh_AEYW}, and the uniform convergence of the scheme is a consequence of the Comparison Principle.\\
 To our knowledge, the approximation of Hamilton-Jacobi equations on stratified domains has never been studied before. Nevertheless, the problem   has some similarities with  Hamilton-Jacobi equations on networks,   where the dynamics and running cost can jump across the vertices. Approximation schemes  for these problems  have been considered for example in \cite{CFS,GK,Morfe}. In particular, in \cite{CFF}, it is considered a semi-Lagrangian scheme which has some similarities with the proposed one.\\
 In this paper,  we pay particular attention to the algorithm for solving the approximation scheme, and to the qualitative analysis of the corresponding numerical tests, revealing new and in some ways unexpected phenomena with respect to the case of continuous Hamilton-Jacobi equations. 
 To this end, we implemented the first release of \texttt{HJSD}, a free and standalone software which is able to solve optimal control problems on stratified domains in two and three dimensions. The program accepts, as input, simple and suitably formatted text files, containing all the relevant information on the stratification and the related control problems. Then it writes in output a standard \texttt{vtk} file containing the computed solution and the optimal controls of the problem, that can be effectively visualized by classical vtk viewers, such as \texttt{Paraview}.\\
 Both theory and numerics presented in this work can be extended in several directions, including more general Hamiltonians, time dependent equations and non flat stratifications approximated by unstructured meshes. We plan to address these topics in a forthcoming paper.
 
The paper is organized as follows. In Section \ref{sec:continuous}, we briefly introduce  the optimal control problem and recall the main theoretical results.   In Section \ref{sec:semi_discrete}, we describe the approximation scheme and prove its convergence to the value function of the optimal control problem. Section \ref{sec:software} contains some information on the \texttt{HJSD} software developed for the solution of Hamilton-Jacobi equations on stratified domains. Finally, in Section \ref{sec:numerics}, we present some numerical tests in two and three dimensions, and further comments on the problem.

\section{Hamilton-Jacobi equations on stratified domains}\label{sec:continuous}
We briefly recall the theory of Hamilton-Jacobi equations on stratified domains which has been developed in \cite{BrYu, BaCh_AEYW} (see also \cite[Chapter 21]{BaCh_book}).  We only consider   the case of a flat stratification, i.e.  when disjoint  submanifolds of $\R^N$ are locally affine subspace.
\begin{defi}
	We say that $\M=(\Man{k})_{k=0,\dots, N}$ is an \textit{Admissible Flat Stratification} (AFS in short)  if, for $k=0,\dots,N$, it holds:
	\begin{itemize}
		\item[(i)]
		For any $x \in \Man{k}$, there exists $r>0$ and $V_k$, a
		$k$-dimensional linear subspace of $\R^N$, such that  $$ B(x,r) \cap
		\Man{k} = B(x,r) \cap (x+V_k)\; .$$ Moreover $B(x,r) \cap
		\Man{l}=\emptyset$ if $l<k$.
		\item[(ii)]
		If $\Man{k} \cap \overline {\Man{l}}\neq \emptyset$ for some $l>k$
		then $\Man{k} \subset \overline {\Man{l}}$.
		\item[(iii)]  $\overline{\Man{k}} \subset
		\Man{0}\cup\Man{1}\cup\cdots\cup\Man{k}$.
	\end{itemize}
\end{defi}
In the case $k=0$, $V_k=\{0\}$ and {\it  (i)}  implies that the set $\Man{0}$, if not
void, consists of isolated points.\\
We   define a  control problem associated to a differential
inclusion.  Let  $\BL:\R^N \to\mathcal{P}(\R^{N+2})$ be a set-valued map satisfying
\begin{itemize}
	\item[{[H0]}] The map $x\mapsto\BL(x)$ has compact, convex images
	and is upper semi-continuous;
	\item[{[H1]}] There exists $M>0$, such that for any $x\in\R^N$,
	$$\BL(x)\subset \big\{(b,c,\ell)\in\R^{N+2}:|b|\leq M;|c|\le M; |\ell| \leq
	M\big\}\,.$$
\end{itemize}
Assumption [H1] implies that  the data of the control problem are uniformly bounded. We consider  the  differential inclusion 
$$\frac{\d}{\dt}(X,C,L)(s)\in\BL\big(X(s)\big)\ \text{ for a.e. }s\in[0,t)\,, 
\quad \text{and }(X,C,L)(0)=(x,0,0)\,.$$
Under the previous assumptions on $\BL$,   given $x\in\R^N$,
there exists a Lipschitz function $(X,C,L):\R^+\to\R^{N+2}$ which is a
solution of this differential inclusion and, for almost any
$s$, $(\dot X,\dot C,\dot L)(s)=(b,c,\ell)(s)$ for some $(b,c,\ell)(s)\in\BL(X(s))$. 
Here $b$, $c$ and $\ell$ represent, respectively the dynamics, the discount factor and the running cost  associated to the  control problem. \\
The value function of the control problem is defined by
$$\Um(x)=\inf_{(X,C,L)\in\cT(x)}\Big\{\int_0^\infty 
e^{-C(t)}L(t)\dt\Big\}\,,
$$
where $\cT(x)$ stands for all the Lipschitz trajectories $(X,C,L)$ of the
differential inclusion which start at $(x,0,0)$
. The Hamiltonian associated to the previous  control problem is given by
\begin{equation}\label{Ham_global_usc}
	H(x,r,p)=\sup_{(b,c,\ell)\in\BL(x)}\big\{ -b\cdot p +c r - \ell \big\}\,.
\end{equation}
The assumptions on the map $\BL$ implies  that $H$ is upper semi-continuous (w.r.t.
all variables) and is convex and Lipschitz   in $p$. The notion of viscosity solution is not sufficient to characterize the value function $\Um$ as the unique viscosity solution of the Hamilton-Jacobi equation
\begin{equation}\label{HJS}
	H(x,u,D u)=0 \quad \text{on}\quad \R^N
\end{equation}
and additional conditions for subsolutions on the discontinuity set have to be introduced. For $k=0,\dots, N$, we introduce the Hamiltonian $H^k$ on $\Man{k}$
\[
H^k(x,r,p)=
\sup_{\genfrac{}{}{0pt}{1}{(b,c,l)\in\BL(x)}{b\in T_x\Man{k}}}
\big\{ -b\cdot p +cr- \ell \big\},\qquad  x\in\Man{k},\, p\in \R^N, r\in\R.
\]
We now give the definition of stratified    solution for the problem \eqref{HJS} (for the notion of viscosity solution, we refer to \cite{BCD,Ba}). 
\begin{defi}
We say that
	\begin{itemize}
		\item[(i)] a locally bounded, lsc function $v:\R^N\to\R$ is a stratified supersolution of \eqref{HJS} if it is a viscosity supersolution of the equation;
		\item[(ii)] a locally bounded, usc function $u:\R^N\to\R$ is a stratified subsolution of \eqref{HJS} if it is a viscosity subsolution of the equation and
		for any $k=0,\dots, N$,  for any test-function $\phi \in
		C^1(\Man{k})$ such that  $u-\phi$ has a local maximum point  at $x \in \Man{k}$  on $\Man{k}$,  then 
		$$H^k (x, u (x),D\phi (x)) \leq  0 \; .$$
			\end{itemize}
\end{defi}
The following result characterizes the value function $\Um$ as the unique stratified solution of \eqref{HJS}. In particular, the uniqueness is consequence of a Comparison Principle among bounded stratified sub and supersolutions of \eqref{HJS}.
\begin{theo}[{\cite[Theorem 21.3.1 and Corollary 22.1.2]{BaCh_book}}] \label{thm:uniq_strati}
	Assume [H0], [H1] and
	\begin{itemize}
		\item[{[H2]}]  for any $0\leq k\leq j\leq N$,   if
		$y_1,y_2\in\Man{j}\cap B(x,r)$ with $y_1-y_2 \in V_k$, then
		$$\begin{cases}\distH\big(\B(y_1),\B(y_2)\big) \leq C_1|y_1-y_2|\,,\\[2mm] 
			\distH\big(\BL(y_1),\BL(y_2)\big) \leq
			\omega\big(|y_1-y_2|\big)\,,\end{cases}$$
where $\B(x)=\big\{b \in\R^N:\ \hbox{there exists $c$, $\ell\in\R$ such that  } (b,c,\ell)\in\BL(x) \big\}$, $\distH$ denotes the Hausdorff distance and $\omega$ is a modulus of continuity.
		\item[{[H3]}] there exists $\delta >0$ such that, for any $0\leq k< N$, if $y\in B(x,r)\setminus\Man{k}$ there holds
		$$B(0,\delta) \cap V_k^\bot \subset \B(y) \cap V_k^\bot\,;$$
		\item[{[H4]}] there exists a constant $\bc>0$ such that, for all $x\in \R^N$, if $(b,c,l)\in \BL(x)$, then $c\ge \bc$.
	\end{itemize}
	Then, the value function $\Um$ is  continuous and it is  the unique  stratified solution of  \eqref{HJS}.
\end{theo}
More general assumptions for the data of the problem  and also the case of a general stratification are discussed in \cite{BaCh_AEYW, BaCh_book}.

\section{A semi-Lagrangian scheme for Hamilton-Jacobi equations on stratified domains}
\label{sec:semi_discrete}
In this section, we consider   an approximation scheme for  the problem described in the previous section. We  consider a specific form for the Hamilton-Jacobi equation that corresponds to assigning an optimal control problem on each submanifold $\Man{k}$, $k=0,\dots, N$ of the stratification  (see \cite[Section 22.2.1]{BaCh_book} for more details).\\
Given an AFS   $\M$ of $\R^N$, we write each manifold as the union of its connected components, i.e.
$\Man{k}=\cup_{j=1}^{J(k)} \Man{k,j}$ with $J(k)\in\N$. On each $\Man{k,j}$, we consider a   control set $A^{k,j}$, a dynamics $b^{k,j}:\R^N\times A^{k,j}\to \R$, with $b^{k,j}(x,\alpha) \in T_x\Man{k}$ for $(x,\alpha)\in \Man{k,j}\times A^{k,j}$, a discount factor $c^{k,j}\in\R$  and a running cost $\ell^{k,j}: \R^N\times A^{k,j}\to \R$. We define the associated Hamiltonian
\begin{equation*} 
	H^{k,j}(x,r,p)=\sup_{\alpha\in A^{k,j}}\Big\{-b^{k,j}(x,\alpha)\cdot p+c^{k,j}r-\ell^{k,j}(x,\alpha)\Big\},\qquad (x,p)\in \Man{k,j}\times T_x\Man{k,j}.
\end{equation*}
Note that for $x\in\Man{0,j}$, $j \in J(0)$, we have
$$H^{0,j}(x,r,p)=\sup_{\alpha\in A^{0,j}}\Big\{-\ell^{0,j}(x,\alpha)+c^{0,j}r\Big\}$$
since $T_x M^{0,j}=\{0\}$.\\
Set $L(x)=\{(k,j):\, x\in \overline{\Man{k,j}}\}$ and consider the map $\BL:\R^N \to\mathcal{P}(\R^{N+2})$ given by 
\[
\BL(x)=\cob\Big\{\bigcup_{(k,j)\in L(x)}\{(b^{k,j}(x,\alpha), c^{k,j},\ell^{k,j}(x,\alpha)):\alpha \in A^{k,j}\}\Big\},
\]
where  $\cob$ denotes the closure of the convex hull, and the   Hamiltonian $H:\R^N\times \R^N\to\R$ given by 
\begin{equation}\label{Ham_global}
	H(x,r,p)=\max_{(k,j)\in L(x)}\,\sup_{\alpha\in A^{k,j}}\Big\{-b^{k,j}(x,\alpha)\cdot p+c^{k,j}r-\ell^{k,j}(x,\alpha)\Big\}.
\end{equation}
Assumptions of Theorem \ref{thm:uniq_strati} are satisfied if
\begin{itemize}
	\item  $b^{k,j}$, $\ell^{k,j}$ are bounded, continuous functions on $\R^N\times A^{k,j}$ and  
	$\{(b^{k,j},\ell^{k,j})(x,\alpha): \, \alpha\in A^{k,j}\}$ are convex, compact subsets of $\R^{N+1}$. Moreover, for any $R>0$, there exists a constant $C_R$ such that 
	\[|b^{k,j}(x,\alpha)-b^{k,j}(y,\alpha)|\le C_R|x-y|\qquad \forall x,y\in B(0,R),\, \alpha\in A^{k,j}.\]
	\item $c^{k,j}\ge \bc$ for some positive constant $\bc$.
	\item For $x\in\Man{{\bar k}}$, the set
	\[\cob\Big\{\bigcup_{(k,j)\in L(x), k>\bar k}\{(b^{k,j}(x,\alpha), \ell^{k,j}(x,\alpha)):\alpha \in A^{k,j}\}\Big\}\]
	satisfies [H3].
\end{itemize}
The previous assumptions will be made in the rest of paper.\par
We now describe a classical semi-Lagrangian approximation scheme for \eqref{HJS}    based on a two-step discretization procedure (see \cite{FalFer}). 
For $h\in (0,1/\bc)$ and $H$ given by \eqref{Ham_global},  we first define the following semi-discrete approximation scheme for \eqref{HJS}
\begin{equation}\label{scheme_S}
	u(x)=\min_{(k,j)\in L(x)}\,\inf_{\alpha\in A^{k,j}}\Big\{(1-c^{k,j} h)u(x+hb^{k,j}(x,\alpha))+h\ell^{k,j}(x,\alpha)\Big\},\qquad x\in\R^N.
\end{equation}
In the second discretization step, we introduce a FEM like discretization of
\eqref{scheme_S} yielding a fully discrete scheme for \eqref{HJS}.
For $\dx>0$, let $\mathcal{T}^\dx=\{S^\dx_\t\}_{\t\in\N}$ be a non-degenerate
triangulation of $\R^N$, i.e. a collection of $N$-simplices $S^\dx_\t$ such that
$$\underset{\t\in\N}{\cup} S^\dx_\t=\R^N,\quad \sup_{\t\in\N}( \hbox{diam}\,S_\t)\leq
\dx,\quad \rho \dx\leq \sup_{\t\in\N}( \hbox{diam}\,B_{S^\dx_\t}),$$
where  $\rho\in(0,1)$, $\hbox{diam}$ denotes the diameter of the set,
and $B_{S^\dx_\t}$ is the greatest ball contained in $S^\dx_\t$.
We denote
by $X^\dx=\{x_\s\}_{\s\in\N}$ the corresponding set of the vertices,
and introduce the space of continuous piecewise linear functions
on $\mathcal{T}^\dx$,
\[
W^\dx=\{w\in C(\R^N) :\text{$Dw(x)$ is constant in $S_\t^\dx$, $\t\in\N$}\}.
\]
Every element $w$ in $W^\dx$ can be expressed as
\begin{align*}
	w(x)=\sum_{\s\in \N}\beta_\s(x)w(x_\s),
\end{align*}
for basis functions $\beta_\s\in W^\dx$
satisfying $\beta_\s(x_\t)=\delta_{\s\t}$ for $\s,\t\in\N$. It immediately
follows that $0\leq \beta_\s(x)\leq 1$, $\sum_{\s\in\N}\beta_\s(x)=1$,
$\beta_\s$ has compact support, and at any $x\in\R^N$ at most $N+1$
$\beta_\s$'s are non-zero.   We require that the triangulation $\mathcal{T}^\dx$   is adapted to the stratification $\Man{k}$, $k=0,\dots, N$, i.e. given $x\in\Man{k}$ with $x=\sum_\s\beta_\s(x)x_\s$, then $\beta_\s(x)\neq 0$ if and only if $x_\s\in \overline{\Man{k}}$. In other words,  $\mathcal{T}^\dx$  induces a triangulation of size $\dx$ also of the flat manifold $\overline{\Man{k}}$. In particular, points of $\Man{0}$ must be vertices of the triangulation.\\
The fully discrete scheme can then be formulated as follows:
Find the function $u\in W^\dx$ that satisfies \eqref{scheme_S} at every
vertex $x_\s\in X^\dx$, or equivalently,
\begin{equation}\label{scheme_fully}
	u(x_\s)=\T(h,\dx,x_\s,u), \qquad x_\s\in X^\dx,
\end{equation}
where
\[\T(h,\dx,x_\s,u)=\min_{(k,j)\in L(x_\s)}\,\inf_{\alpha\in A^{k,j}}\Big\{(1- c^{k,j} h)
\sum_{\t\in \N}I^{k,j}_{\s\t}(\alpha)u(x_\t) 	+h\ell^{k,j}(x_\s,\alpha)\Big\}\]
and the interpolation matrix $I^{k,j}(\alpha)$ is given
by
\begin{equation}\label{matrix}
 I^{k,j}_{\s\t}(\alpha)=\beta_\t\big(x_\s+h
b^{k,j}(x_\s,\alpha)\big).   
\end{equation}
Note that  only $N+1$ entries of any row of $I^{k,j}(\alpha)$ are non-zero and $\sum_{\t\in \N}I^{k,j}_{\s\t}(\alpha)=1$.
%
%
%
\begin{pro}\label{prop:scheme_S}
	There exists a unique bounded solution $\uh\in W^\dx$ of \eqref{scheme_fully}.	
\end{pro}
\begin{proof}
For $u,v\in W^\dx$, we have that 
	\begin{equation*}\label{map_contraS}
		\begin{split}
			|\T(h,\dx,x_\s,u)-\T(h,\dx,x_\s,V)|\le (1-\bc h)\|u-v\|_\infty
		\end{split}
	\end{equation*}
	for all $x_\s\in X^\dx$. The previous  inequality implies that $\T$ is a contraction on the set $W^\dx$, hence    there exists a unique bounded solution $\uh\in W^\dx$ to \eqref{scheme_fully}.\\
Given $M$ as in [H1], we have $\|\ell^{k,j}\|_\infty \le M$ for any $k=0,\dots, N$, $j\in J(k)$.
Hence
\begin{align*}  		
	&\uh(x)=\T(h,\dx,x_\s,u)\le \min_{(k,j)\in L(x)}\,\inf_{\alpha\in A^{k,j}} (1-c^{k,j} h)\sum_{\t\in \N}I^{k,j}_{\s\t}(\alpha)\|\uh\|_\infty+h\|\ell^{k,j}\|_\infty \\
	&\le (1-\bc h)\|\uh\|_\infty+hM.
\end{align*}  
and therefore $\|\uh\|_\infty\le M/\bc$.
\end{proof}
\begin{rem}
    For $x\in \Man{0,j}$, $j \in J(0)$, taking into account that $x\in X^\dx$, the scheme \eqref{scheme_fully} implies that $$\uh(x)\le\frac{1}{\bc} \inf_{\alpha\in A^{0,j}}\{\ell^{0,j}(x,\alpha)\}$$
\end{rem}
For the convergence analysis, we rewrite the fully discrete scheme \eqref{scheme_fully} in a equivalent way as
\begin{equation*}
	\Sigma(h,\dx, x_\s ,u(x_\s),u)=0 \qquad \s\in \N
\end{equation*}
where, for  $x_\s\in X^\dx$, $r\in\R$,  $w\in W^\dx$, we have
\begin{equation} \label{scheme_S1}
	\Sigma(h,\dx, x_\s,r,w)=\max_{(k,j)\in L(x_\s)}	\Sigma^{k,j}(h,\dx, x_\s,r,w)
\end{equation}
with
\begin{equation}\label{scheme_local}
\begin{split}
   \Sigma^{k,j}(h,\dx, x_\s,r,w)=c^{k,j} r+\sup_{\alpha \in A^{k,j}}\Big\{-(1-c^{k,j} h)\frac{w(x_\s+hb^{k,j}(x_\s,\alpha))-r}{h}-\ell^{k,j}(x_\s,\alpha)\Big\}\\
   =c^{k,j} r+\sup_{\alpha \in A^{k,j}}\Big\{-(1-c^{k,j} h)\frac{\sum_{\t\in\N}I^{k,j}_{\s\t}(\alpha)w(x_\t)-r}{h}-\ell^{k,j}(x_\s,\alpha)\Big\}
\end{split}
\end{equation}
and $I$ as in \eqref{matrix}.
In the following, given a smooth function $\phi\in C^\infty_b(\R^N)$,   we denote with $\hphi$ its linear interpolation   on the simplices  $S^\dx_\t$ of the triangulation.
\begin{pro}\label{prop:prop_scheme_S}
The scheme  \eqref{scheme_S} satisfies the following properties
	\begin{itemize}
			\item[(P1)] The scheme  is invariant by addition of constants, i.e.
			\[	\Sigma(h,\dx, x_\s,r+C,w+C)=\S(h,x_\s,r,w),\]
for any $x_\s\in X^\dx$,  $ C\in\R$, $w\in W^\dx$.
			\item[(P2)] The scheme  is monotone, i.e. given two bounded functions $v$, $w\in W^\dx$ such that $v\le w$, then
			\[\Sigma(h,\dx, x_\s,r,w)\le \Sigma(h,\dx, x_\s,r,v),\]
			for any $x_\s\in X^\dx$.
			\item[(P3)] 
	For any smooth function $\phi\in C^\infty_b(\R^N)$ and   $x_\s\in X^\dx$, 
	\begin{equation}\label{est_conv}
	|H(x_\s,\phi(x_\s),D\phi(x_\s))-	\Sigma(h,\dx, x_{\s},\hphi(x_{\s}),\hphi)|\le  C_\phi \left(h+\frac{\dx}{h}\right), 	
	\end{equation}
for some constant $C_\phi=C(\|\phi\|_{C^2})$.
\end{itemize}
\end{pro}	
\begin{proof}
Properties  (P1) and (P2) are immediate. To prove (P3),  for $(h,x,r)\in  (0,1/\bc)\times\R^N\times \R$ and a bounded smooth function $\phi$, we set
\begin{equation*} 
	\S(h,x,r,\phi)=\max_{(k,j)\in L(x)}	\S^{k,j}(h,x,r,\phi)
\end{equation*}
with
\begin{equation*}
	\S^{k,j}(h,x,r,\phi)=c^{k,j} r+\sup_{\alpha \in A^{k,j}}\Big\{-(1-c^{k,j} h)\frac{\phi(x+hb^{k,j}(x,\alpha))-r}{h}-\ell^{k,j}(x,\alpha)\Big\}.
\end{equation*}
Then, for all $x\in\R^N$  
	\begin{align*} 
					&	|H(x,\phi(x), D\phi) -\S(h,x,\phi (x),\phi)|\\
					&\le\sup_{x\in\R^N} \max_{(k,j)\in L(x)}	| H^{k,j}(x,\phi(x),D\phi)-\S^{k,j}(h,x,\phi (x),\phi)|\\
					&\le C h,
		\end{align*}
for some constant $C=C(\|D^2\phi\|_\infty)$. Moreover, since $\|\phi-\hat\phi\|_\infty\le C\dx$, if  $x_\s\in X^\dx$  we have
\begin{align*}
&	|\S(h,x_\s,\phi(x_\s),\phi)-\Sigma(h,\dx, x_{\s},\hphi(x_{\s}),\hphi)|\\
&	\le \sup_{x\in\R^N} \max_{(k,j)\in L(x)}	|\S^{k,j}(h,x_\s,\phi(x_\s),\phi)-\Sigma^{k,j}(h,\dx, x_{\s},\hphi(x_{\s}),\hphi)|\\
&	\le C\frac{\dx}{h}.
\end{align*}
Taking into account  the previous inequalities, we get
\[
| H(x_\s,\phi(x_\s),D\phi(x_\s))-	\Sigma(h,\dx, x_{\s},\hphi(x_{\s}),\hphi)|
	\le C \left(h+\frac{\dx}{h}\right)
\] 
with $C=C(\phi)$  independent of $x_\s$.
\end{proof}


\begin{theo}\label{thm:convS}
	 If $h,\dx\to 0$, with $\dx/h\to 0$, the sequence $\uh$ of the solutions  to \eqref{scheme_S1} converges to the unique stratified solution $\Um$ of \eqref{HJS}.
\end{theo}
\begin{proof}
	The proof is based again on the classical technique of half-relaxed limits introduced in \cite{BaSo}. We define 
	\begin{equation*}
		\underline{u}(x):={\liminf_{\genfrac{}{}{0pt}{1}{h,\dx\to 0 }{y\to x}}}{}_{ *}  \: \uh(y),  \quad \quad \overline{u}(x):={\limsup_{\genfrac{}{}{0pt}{1}{h,\dx\to 0 }{y\to x}}}{}^{*}   \:  \uh(y) \: 
	\end{equation*}
and we denote by $H^*$, $H_*$ the upper and lower semicontinuous envelopes of the Hamiltonian $H$ (see \cite{Ba} for the previous definitions).\\
We claim that $\underline{u}$, $\overline{u}$ are a supersolution and, respectively, a subsolution to \eqref{HJS} in the standard viscosity sense. Indeed, by \eqref{est_conv}, for any $x\in\R^N$, for any smooth function $\phi$ and for any sequence $x_\s\to x$ for $h,\dx\to 0$ with $ \dx/h\to 0$, we have
\begin{align*}
	&\limsup_{h,\dx\to 0 } \, \Sigma(h,\dx, x_{\s},\hphi(x_{\s}),\hphi)\\
	&\le \limsup_{h,\dx\to 0 } \left[H(x_\s, \phi(x_\s),D\phi(x_\s))+C_\phi \left(h+\frac{\dx}{h}\right)\right]\\
	&\le H^*(x, \phi(x_\s), D\phi(x))=H(x,\phi(x),D\phi(x))				
\end{align*}
(recall that $H$ is usc) and similarly
\[\limsup_{h,\dx\to 0 } \, \Sigma(h,\dx, x_{\s},\hphi(x_{\s}),\hphi)\ge   H_*(x,\phi(x),D\phi(x)).\]	
Hence, the scheme \eqref{scheme_S1} is consistent, in viscosity solution sense, with \eqref{HJS} and the claim is a direct consequence of the convergence result in \cite[Theorem 2.1]{BaSo}.\\
If we show that $\overline{u}$ is a stratified subsolution, then the comparison principle (see \cite[Theorem 5.2.]{BaCh_AEYW} and \cite[Theorem 21.3.1]{BaCh_book}) implies $\overline{u}\le \underline{u}$.
	Being the reverse inequality straightforward and taking into account  Theorem \ref{thm:uniq_strati}, we get $\overline{u}= \underline{u}=\Um$,  and therefore the local uniform convergence of $\uh$ to $\Um$. Hence, it remains to show that  for any smooth  test function $\phi $ such that  $\overline{u}-\phi$ has a local (strict) maximum point  at $x_0 \in \Man{k}$  on $\Man{k}$ with $\overline{u}(x_0)=\phi(x_0)$,  then 
	\begin{equation}\label{sub_stratified}
	 	H^k (x_0,\phi(x_0), D\phi (x_0)) \leq  0 \; .
	\end{equation}
Since the argument is local,   we can consider a neighborhood of $x_0$ small enough   in such a way that it contains no point of $\Man{l}$ for $l<k$ and of no other connected component  of $\Man{k}$ except   the one, denoted with $\Man{k,p}$, containing $x_0$. \\
Since $x_0$ is a maximum point for $\overline{u}-\phi$ on $\Man{k}$, there exists a sequence of maximum points $\xh\in\Man{k}$ for $\uh-\phi$ such that $\xh\to x_0$ and $\uh(\xh)\to \overline{u}(x_0)$  for $h,\dx\to 0$. To simplify the notation, we write $\uh=u_\Delta$, $\xh=\xd$, where  $\Delta\to 0$ for $h,\dx\to 0$. Since the triangulation is adapted to the stratification, we have $\xd=\sum_\s\beta_\s(\xd)x_{\Delta,\s}$ with $\beta_\s(\xd)\neq 0$ if and only if $x_{\Delta,\s}\in \Man{k}$ and, moreover,
$\xdj+hb^{k,p}(\xdj,\alpha)\in \Man{k,p}$ for $h$ sufficiently small. Hence
\[
\phi(\xdj+hb^{k,p}(\xdj,\alpha))-\phi(\xd)\ge
\ud(\xdj+hb^{k,p}(\xdj,\alpha))-\ud(\xd).
\]
Recalling \eqref{scheme_S1} and that $\sum_\s\beta_\s(\xd)=1$, we have 
\begin{equation}\label{stima0}
	\begin{split}
0&=\sum_\s\beta_\s(\xd)\Sigma(h,\dx, \xdj,\ud(\xdj),\ud)\\
&\ge 	\sum_\s\beta_\s(\xd) \Sigma^{k,p}(h,\dx, \xdj,\ud(\xdj),\ud)\\
&\ge c^{k,p} \ud(\xd)+\sup_{\alpha \in A^{k,p}}\Big\{-(1-c^{k,p} h)\Big[\frac{\sum_\s\beta_\s(\xd)\phi(\xdj+hb^{k,p}(\xdj,\alpha))-\phi(\xd)}{h}\Big]\\
&-\sum_\s\beta_\s(\xd)\ell^{k,p}(\xdj,\alpha)\Big\}.
\end{split}
\end{equation}
We estimate
\begin{align}
&\Big|\frac{1}{h}\sum_\s(\beta_\s(\xd)\phi(\xdj+hb^{k,p}(\xdj,\alpha))-\phi(\xd))-	D\phi (\xd)\cdot b^{k,p}(\xd,\alpha)\Big|\le C\frac{\dx}{h}+\omega(\dx)\label{stima1}\\
&|\sum_\s\beta_\s(\xd)\ell^{k,p}(\xdj,\alpha)-\ell^{k,p}(\xd,\alpha)|\le \omega(\dx)\label{stima2}
\end{align}
 for any $\Delta$, with $C=C(\phi)$ and  $\omega$ the continuity modulus of $b^{k,p}, \ell^{k,p}$.
Replacing \eqref{stima1}	and \eqref{stima2} in \eqref{stima0}, we get
\begin{align*}
	0&\ge c^{k,p} \ud(\xd)+ \sup_{\alpha\in A^{k,p}}\Big\{-b^{k,p}(\xd,\alpha)\cdot D \phi(\xd)-\ell^{k,p}(\xd,\alpha)\Big\}+o(1)\\
	&=H^k(\xd, \ud(\xd), D\phi(\xd) )+o(1),
\end{align*}
where $o(1)\to 0$ for $h,\dx\to 0$ with $\dx/h\to 0$. 
Hence, passing to the limit, we get \eqref{sub_stratified} and the claim is proved.

\end{proof}
\begin{rem}	
	If the Hamiltonian is globally defined in $\R^N$, see \eqref{Ham_global_usc},  then  the semi-discrete scheme reads as
	\[
	u(x)=\inf_{(b,c,\ell)\in\BL(x)}\{-(1-c  h)u(x+hb)-h\ell\}.
	\]
	and the fully discrete one corresponds to find a function $u\in W^\dx$ that 	satisfies the previous equation at every
	vertex $x_\s\in X^\dx$, or equivalently,
	\begin{align*} 
		u(x_\s)=\inf_{(b,c,\ell)\in\BL(x_\s)}\Big\{(1-c h)
		\sum_{\t\in \N}I^{k,j}_{\s\t}(b)u(x_\t) 	+h\ell\Big\}, \qquad x_\s\in X^\dx
	\end{align*}
	where the matrix $I^{k,j}(b)$ is given
	by
	\[I^{k,j}_{\s\t}(b)=\beta_\t\big(x_\s+h 	b\big).\]	
	The convergence analysis in Prop. \ref{prop:prop_scheme_S} and Theorem \ref{thm:convS} can be easily extended to this case.
\end{rem}
\begin{rem}
	In \cite{BBC1,BBCI}, it is characterized  the value function of a control problem where the dynamics and running cost have a discontinuity along a set of co-dimension one. More precisely, using the notation of those papers,  consider
	\begin{align*}
		\Omega_1:=\{x=(x_1,\cdots,x_N); x_N>0\},\\
		\Omega_2:=\{x=(x_1,\cdots,x_N); x_N<0\},\\
		\H:= \overline{\Omega}_1 \cap
		\overline{\Omega}_2 = \left\{  x \in \R^N  \: : \: x_N=0 \right\}.
	\end{align*}
	and the Hamiltonians
	\begin{align*}  
		&H_1(x,r,p):=\sup_{ \alpha_1 \in A_1} \left\{  -b_1(x,\alpha_1) \cdot p+ r  - \ell_1(x,\alpha_1) \right\}\,,	\\  
		&H_2(x,r,p):=\sup_{ \alpha_2 \in A_2} \left\{  -b_2(x,\alpha_2) \cdot p+r  - \ell_2(x,\alpha_2) \right\}\,,
	\end{align*}    
	$$ H(x,r,p):= \left\{\begin{array}{ll}H_1(x,r,p) & \hbox{if $x\in \Omega_1$},\\[10pt]
		H_2(x,r,p) & \hbox{if $x\in \Omega_2$}.\end{array}\right.$$ 
	Since the Hamiltonian $H$ has a discontinuity along the hyperplane $\H$, the  standard definition of viscosity sub and supersolutions  for 
	\[
	H(x, u ,D u)=0 \quad \text{on}\quad \R^N 
	\]
	reads as (see \cite{BCD,Ba})
	\begin{equation}\label{eqn-minus}
		\begin{cases}
			  H_1(x,u,Du) = 0     &   \hbox{ in   }\Omega_1 \,, \\[10pt]
			   H_2(x,u,Du) = 0     &   \hbox{ in   }\Omega_2\,, \\[10pt]
			\min\{   H_1(x,u,Du),   H_2(x,u,Du)\}\leq0  &\hbox{on   } \H \;, \\[10pt]
			\max\{   H_1(x,u,Du),  H_2(x,u,Du)\}\geq 0 &  \hbox{on   } \H.   
		\end{cases} 
	\end{equation}
	The conditions induced  on $\H$ by \eqref{eqn-minus} are not enough to ensure uniqueness of the solution and there may exists several viscosity solution of the problem. In \cite{BBC1,BBCI,BaCh_book}, different definitions to characterize the value function as the unique solution of the previous problem are introduced. Another possibility is to interpret \eqref{eqn-minus} in the stratified sense, with a stratification given by   $\Man{N-1}=\H$, $\Man{N}=\R^N\setminus \H$. Introduce the tangential Hamiltonian
	\begin{align}\label{def:HamHT} 
		\HT(x,r,p):=\sup_{a\in A_\H(x)} \big\{  -b_\H(x,\a)\cdot p_\H +r -  l_\H(x,\a)   \big\}\,,
	\end{align} 
	with  
	\begin{align*} 
		&	b_\H\big(x,a)=b_\H\big(x,(\alpha_1, \alpha_2, \mu)\big):=\mu
		b_1 (x,\alpha_1) + (1-\mu)b_2(x,\alpha_2)\,,\\[10pt]  
		&	\ell_\H(x,\a)=\ell_\H\big(x,(\alpha_1, \alpha_2, \mu)\big):=\mu
		\ell_1 (x,\alpha_1) + (1-\mu)\ell_2(x,\alpha_2)\,, \\[10pt]
		&	A_\H(x):=\Big\{(\alpha_1, \alpha_2, \mu)\in A_1\times A_2\times [0,1]:
		b_\H\big(x,(\alpha_1, \alpha_2, \mu)\big)\cdot
		e_N(x) = 0\Big\}\, .
	\end{align*}
	and consider on $\H$ the Hamilton-Jacobi equation  
	\begin{equation}\label{sub_H}
		\HT(x,u, D_\H u)= 0\quad \text{on}\quad \H.
	\end{equation}
	Then,  Theorem \ref{thm:uniq_strati} implies that $\Um$ is the unique stratified solution of \eqref{HJS}
	with
	$H^{N-1}(x,r,p)=H_T(x,r,p)$ for $x\in\H$. The fully discrete approximation scheme \eqref{scheme_fully} reads in this case as 
	\[
	u(x_\s)=
	\begin{cases}
		\dis{\inf_{\alpha\in A_i}\{-(1-  h)u(x_\s+hb_i(x_\s,\alpha))+h\ell_i(x_\s,\alpha)\}}&x_\s\in \Omega_i,\,i=1,2,\\[10pt]
		\dis{\min_{i=1,2}\,\inf_{\alpha\in A_i}\{  -(1-  h)u(x_\s+hb_i(x_\s,\alpha))+h\ell_i(x_\s,\alpha)\}}&x_\s\in \H.
	\end{cases}
	\]
\end{rem} 
\section{The software \texttt{HJSD}}\label{sec:software}
In this section we introduce
\texttt{HJSD} (Hamilton-Jacobi on Stratified Domains), an easy-to-use program written in C++ we developed
for the numerical solution of Hamilton-Jacobi equations on stratified domains in two and three dimensions (the free software will be soon available at \texttt{http://www.sbai.uniroma1.it/~fabio.camilli/HJSD.html}).

\texttt{HJSD} is able to solve equations with Hamiltonians  of the form \eqref{Ham_global}, where a specific control problem can be defined on each connected component of the stratification. In this first version of the software, it is assumed that each pair (dynamics, running cost) is of the form $(b(x)a,\ell(x))$, namely both the speed function $b$ and the running cost $\ell$ only depend on the state but not on the control. On the other hand, the control $a$ takes discrete values in $A^1=[-1,1]$ or $A^2=\mathcal S^1=\{a\in\mathbb{R}^2\,|\,|a|=1\}$ or $A^3=\mathcal S^2=\{a\in\mathbb{R}^3\,|\,|a|=1\}$, according to the dimension of the corresponding submanifold. Finally, the computational domain is a box in two or three dimensions, discretized respectively by structured triangles or tetrahedra.\\
\texttt{HJSD} takes in input \texttt{.hjsd} files, which are simple text files containing all the relevant information for the definition of the stratification and of the corresponding control problems. More precisely, \texttt{.hjsd} files are formatted as follows.\\\\
For two dimensional problems, we have the syntax:\\\\
\texttt{
	\#HJSD2D $N_x$ $N_y$ $x_{\min}$ $x_{\max}$ $y_{\min}$ $y_{\max}$ $N_{A_1}$ $N_{A_2}$\\ 
	\#P $x^P$ $y^P$ $\ell^P$ $c^P$\\
	$\dots$ (other points)\\
	\#LX $x^{LX}$ $y_0^{LX}$ $y_1^{LX}$ $b^{LX}(x,y)$ $\ell^{LX}(x,y)$ $c^{LX}$\\
	$\dots$ (other X-lines)\\
	\#LY $y^{LY}$ $x_0^{LY}$ $x_1^{LY}$ $b^{LY}(x,y)$ $\ell^{LY}(x,y)$ $c^{LY}$\\
	$\dots$ (other Y-lines)\\
 	\#S $x^S$ $y^S$ $b^S(x,y)$ $\ell^S(x,y)$ $c^S$\\
	$\dots$ (other surfaces)\\\\
}
meaning that\\\\
$(i)$ \texttt{\#HJSD2D} defines a stratification of $\R^2$ using the computational box $[x_{\min},x_{\max}]\times[y_{\min},y_{\max}]$, discretized by means of a uniform structured grid with $N_x\times N_y$ nodes. Each cell is divided into two triangles to reconstruct the solution via linear interpolation. Moreover, the control sets $A^1$ and $A^2$ are uniformly discretized with $N_{A_1}$ and  $N_{A_2}$ elements respectively.\\\\
$(ii)$ \texttt{\#P} defines a connected component of the submanifold $\mathbf{M}^0$, namely a point with coordinates $(x^P,y^P)$ (actually the closest grid point). Moreover, $\ell^P$ and $c^P$ are, respectively, the constant running cost and discount factor corresponding to the point (no dynamics on zero-dimensional submanifolds, since the related tangent space is trivial).\\\\
$(iii)$ \texttt{\#L(X)(Y)} defines a connected component of the submanifold $\mathbf{M}^1$, namely a coordinate (X)(Y)-line, 
specified by a constant coordinate in one direction and two coordinates in the remaining direction (the line is actually projected on the grid). Note that, if one or both end-points fall out the computational box, the line is considered respectively as semi-infinite and infinite, otherwise it is a segment. The end-points are then excluded in order to obtain an approximation of an open set, according to the definition of $\mathbf{M}^1$: it is up to the user to possibly define the end-points as elements of $\mathbf{M}^0$. Moreover, $b^{L(X)(Y)}(x,y)$ and $\ell^{L(X)(Y)}(x,y)$ are symbolic functions of the variables $x,y$ (parsed and evaluated using the free library \emph{ExprTk} \cite{EXPRTK}) defining, respectively, the speed for the dynamics and the running cost corresponding to the line, while $c^{L(X)(Y)}$ is its discount factor. \\\\
$(iv)$ \texttt{\#S} defines a connected component of the submanifold $\mathbf{M}^2$, namely a flat open surface obtained using a classical Flood Fill algorithm \cite{FLOOD}, specifying a belonging point of coordinates $(x^S,y^S)$. Note that the user should specify the belonging points for all the surfaces induced by the lines in the stratification. Moreover, $b^S(x,y)$ and $\ell^S(x,y)$ are symbolic functions of the variables $x,y$ defining, respectively, the speed for the dynamics and the running cost corresponding to the surface, while $c^S$ is its discount factor. \\\\
Similarly, for three-dimensional problems, we have the following syntax:\\\\
\texttt{
	\#HJSD3D $N_x$ $N_y$ $N_z$ $x_{\min}$ $x_{\max}$ $y_{\min}$ $y_{\max}$ $z_{\min}$ $z_{\max}$ $N_{A_1}$ $N_{A_2}$ $N_{A_3}$\\ 
	\#P $x^P$ $y^P$ $z^P$ $\ell^P$ $c^P$\\
	$\dots$ (other points)\\
	\#LXY $x^{LXY}$ $y^{LXY}$ $z_0^{LXY}$ $z_1^{LXY}$ $b^{LXY}(x,y,z)$ $\ell^{LXY}(x,y,z)$ $c^{LXY}$\\
	$\dots$ (other XY-lines)\\
	\#LYZ $y^{LYZ}$ $z^{LYZ}$ $x_0^{LYZ}$ $x_1^{LYZ}$ $b^{LYZ}(x,y,z)$ $\ell^{LYZ}(x,y,z)$ $c^{LYZ}$\\
	$\dots$ (other YZ-lines)\\
	\#LXZ $x^{LXZ}$ $z^{LXZ}$ $y_0^{LXZ}$ $y_1^{LXZ}$ $b^{LXZ}(x,y,z)$ $\ell^{LXZ}(x,y,z)$ $c^{LXZ}$\\
	$\dots$ (other XZ-lines)\\
	\#SX $x^{SX}$ $y_0^{SX}$ $y_1^{SX}$ $z_0^{SX}$ $z_1^{SX}$ $b^{SX}(x,y,z)$ $\ell^{SX}(x,y,z)$ $c^{SX}$\\
	$\dots$ (other X-surfaces)\\
	\#SY $y^{SY}$ $x_0^{SY}$ $x_1^{SY}$ $z_0^{SY}$ $z_1^{SY}$ $b^{SY}(x,y,z)$ $\ell^{SY}(x,y,z)$ $c^{SY}$\\
	$\dots$ (other Y-surfaces)\\
	\#SZ $z^{SZ}$ $x_0^{SZ}$ $x_1^{SZ}$ $y_0^{SZ}$ $y_1^{SZ}$ $b^{SZ}(x,y,z)$ $\ell^{SZ}(x,y,z)$ $c^{SZ}$\\
	$\dots$ (other Z-surfaces)\\
	\#V $x^V$ $y^V$ $z^V$ $b^V(x,y,z)$ $\ell^V(x,y,z)$ $c^V$\\
	$\dots$ (other volumes)\\\\
}
with the following differences with respect to the two-dimensional case:\\\\
$(i)$ \texttt{\#HJSD3D} defines a stratification of $\R^3$ using the computational box $[x_{\min},x_{\max}]\times[y_{\min},y_{\max}]\times[z_{\min},z_{\max}]$, with additional $N_z$ nodes in the $z$ direction. Each cell of the discretization is divided into six tetrahedra to reconstruct the solution via linear interpolation. Moreover, the control set $A^3$, i.e. the unit sphere, is approximated in spherical coordinates with $N_{A_3}$ meridians and $N_{A_3}$ parallels, with a total number of $N_{A_3}\times N_{A_3}$ elements.\\\\
$(ii)$ \texttt{\#P} and  \texttt{\#L(XY)(YZ)(XZ)} define points and coordinate lines in three-dimensions as before, using additional $z$ coordinates. \\\\
$(iii)$ \texttt{\#S(X)(Y)(Z)} defines a connected component of the submanifold $\mathbf{M}^2$, namely a coordinate (X)(Y)(Z)-plane, specified by a constant coordinate in one direction and two pairs of coordinates in the remaining directions (the plane is actually projected on the grid). As for lines, such flat surfaces are considered semi-infinite, infinite or rectangles, depending on which parts of them fall out the computational box. Again, their boundaries are excluded in order to obtain approximations of open sets, according to the definition of $\mathbf{M}^2$, and it is up to the user to define the missing elements of $\mathbf{M}^0$ and $\mathbf{M}^1$. 
\\\\
$(iv)$ \texttt{\#V} defines a connected component of the submanifold $\mathbf{M}^3$, namely a volume obtained using again a classical Flood Fill algorithm in three-dimension, specifying a belonging point of coordinates $(x^V,y^V,z^V)$.\\\\ 
$(v)$ all the functions defining the different speeds for the dynamics and the running costs on the stratification, are now symbolic functions of the variables $x,y,z$.\\\\
The program checks for syntax errors in the input \texttt{.hjsd} file, then it computes an approximation of the stratified viscosity solution of the problem, and writes the results in a \texttt{vtk} file, ready for the visualization in \texttt{Paraview} \cite{PARAVIEW}. The \texttt{vtk}  file contains the geometry of the stratification, the value function and the corresponding vector field of the optimal dynamics, which can be also used to compute the optimal trajectories starting from arbitrary points (directly in \texttt{Paraview}, by means of the \texttt{StreamTracerWithCustomSource} filter). 

We now provide some details on the actual implementation of \texttt{HJSD}. The code employs the fully-discrete semi-Lagrangian scheme in fixed-point form introduced in \eqref{scheme_fully}. The user can specify at runtime the discretization step $h$ and the tolerance $\tau$ for the convergence. The most delicate part of the algorithm is the construction of a suitable data structure to store a general stratification. More precisely, we have to build, for each grid node $x_\sigma$, the map $L(x_\sigma)$ of all the indices $(k,j)$ of sub-manifolds and related connected components that compete to the node, and we have to attach to each pair $(k,j)$ the corresponding symbolic functions for the dynamics, running cost and discount factor. Then, the evaluation of the fixed-point operator is performed by exhaustively selecting the optimal value among all the $(k,j)$'s and the corresponding controls in $A^{k,j}$. We recall that the reconstruction of the solution through the interpolation  matrix $I^{k,j}$ involves the values at the vertices of a suitable element: the one containing the foot of the characteristic associated to the control $a$, starting from the node $x_\sigma$ and moving for a single time step $h$ along the direction provided by the vector field $b(x_\sigma)a$. Whenever the characteristic falls out the computational box, the corresponding control is penalized so that the optimal one will always point inside the domain. This procedure is straightforward on structured triangular/tetrahedral grids, as the ones considered here, but it can be implemented efficiently also on unstructured grids, see e.g. \cite{CACACE-FERRETTI-22}. Moreover, the computation is intrinsically parallel, since each node can be assigned to a single processor, and a single synchronization is required at each fixed-point iteration. This effectively mitigates the computational effort due to the exhaustive search of the minimum value in the solution update.\\
A CUDA-GPU version of the code is currently under development, including general stratifications with arbitrary (non coordinate) lines and planes on adapted unstructured grids, more general Hamiltonians as in \eqref{Ham_global}, and also suitable descent methods for the computation of optimal controls. 
\section{Numerical experiments}\label{sec:numerics}
In this section, we report and discuss the numerical results obtained by the software \texttt{HJSD}, showing the features of the proposed method in different scenarios. 

Let us start by defining the setup for some parameters common to all the experiments, the other ones will be tuned ad-hoc for the single tests.\\ 
For the first three simulations in two dimensions, we choose the computational box  $[x_{\min},x_{\max}]\times[y_{\min},y_{\max}]=[-1,1]^2$, discretized by $N_x\times N_y=201\times201$ nodes, whereas the control sets $A^1$ and $A^2$ are respectively approximated by $N_{A_1}=3$ and $N_{A_2}=64$ elements. We also fix the tolerance for the convergence $\tau=10^{-6}$, while the discretization step is chosen of order $h=\mathcal{O}(\sqrt{\Delta x})$ to guarantee the convergence of the scheme (Theorem \ref{thm:convS}).\\
On the other hand, for the last simulation in three dimensions, we choose  $[x_{\min},x_{\max}]\times[y_{\min},y_{\max}]\times[z_{\min},z_{\max}]=[-1,1]^3$, $N_x\times N_y\times N_z=101\times101\times 101$, $N_{A_1}=3$, $N_{A_2}=32$ and $N_{A_3}=32$.\\  
Finally, in all the experiments and for visualization convenience, we adopt the convention that the dynamics speed on the zero dimensional submanifold is depicted with a zero value, despite no dynamics is defined on $\mathbf{M}^0$ at all.

\noindent{\bf Test 1.} We consider the stratification of $\R^2$ obtained combining a horizontal segment, its end-points and an additional isolated point. More precisely, we set 
$$\mathbf{M}^0=\{P_0\}\cup\{ P_{1,0}\}\cup \{P_{1,1}\}$$
with
$$P_0=\left(0,\frac34\right),\quad P_{1,0}=\left(-\frac12,0\right),\quad P_{1,1}=\left(\frac12,0\right)\,,$$  $$\mathbf{M}^1=L:=\left(-\frac12,\frac12\right)\times\{0\}\quad\mbox{ and } \quad \mathbf{M}^2=S:=\R^2\setminus(\overline{\mathbf{M}^0\cup\mathbf{M}^1})\,.$$
On each connected component of these submanifolds, we define the dynamics, the running cost and the discount factor according to the following table:
\begin{center}
\begin{tabular}{|c|c|c|c|}
\hline
     Connected component & Dynamics & Running Cost & Discount factor  \\
     \hline
\rule{0pt}{2.5ex}     $P_0$ & - & $\ell^{P_0}=0$ & $c^{P_0}=1$\\
      \hline
 \rule{0pt}{2.5ex}    $P_{1,i}\quad i=0,1$ & - & $\ell^{P_{1,i}}=2$ & $c^{P_{1,i}}=10^{-4}$\\
      \hline
 \rule{0pt}{2.5ex} $L$ & $b^{L}(x,y)\equiv 1$ & $\ell^{L}(x,y)=\frac14(1+4|x|)$ & $c^{L}=10^{-4}$\\
 \hline
 \rule{0pt}{2.5ex} $S$ & $b^{S}(x,y)\equiv 1$ & $\ell^{S}(x,y)\equiv 5$ & $c^{S}=10^{-4}$\\
 \hline
\end{tabular}
\end{center}
This stratification is reported in Figure \ref{Test1}-a, where the color-map represents the values of the running cost on the different submanifolds (while the dynamics, excluding $\mathbf{M}^0$, has everywhere a unitary speed). In practice, the point $P_0$ is an absolute minimizer for the global running cost, and it acts as a target point. On the other hand, the segment $L$ has a strictly positive running cost, ranging in the interval $(\frac14,\frac34)$, and a local minimum at its middle point. Finally, the points $P_1$, $P_2$ and the remaining surface $S$ have an higher and higher running cost.\\ 
In Figure \ref{Test1}-b we report the computed optimal dynamics, while Figure \ref{Test1}-c shows the value function, its level sets and some optimal trajectories. We remark that all the optimal trajectories are forced to eventually reach $P_0$, as in a minimum time problem, and this is implied by the fact that, with the exception of $c^{P_0}$, all the discount factors are close to zero. Nevertheless, we observe that, due to its much more favorable cost, the segment $L$ can temporarily attract some trajectories, also increasing their travelled distance. This results in the creation of a shock which splits the domain in two parts, connected only in a neighborhood of the middle point of $L$, see the surface of the solution in Figure \ref{Test1}-d.\\
Finally, we observe that the contribution of $P_1$ and $P_2$ is irrelevant, both control problems on them are replaced by the one on $L$, extended to its closure as described in Section \ref{sec:semi_discrete}, which is again more favorable.   
\begin{figure}[!h]
\centering{
\begin{tabular}{cc}
	\includegraphics[width=0.49\textwidth]{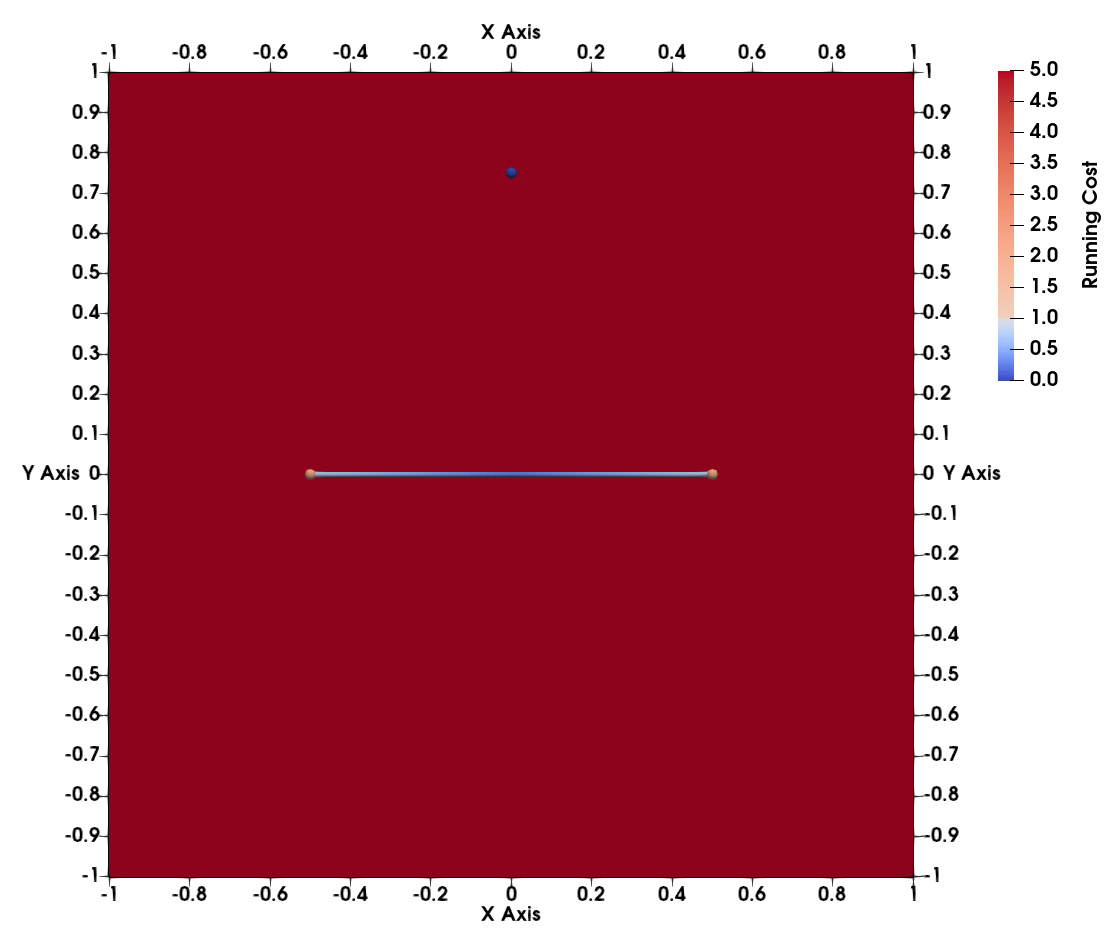} 
	&
	\includegraphics[width=0.49\textwidth]{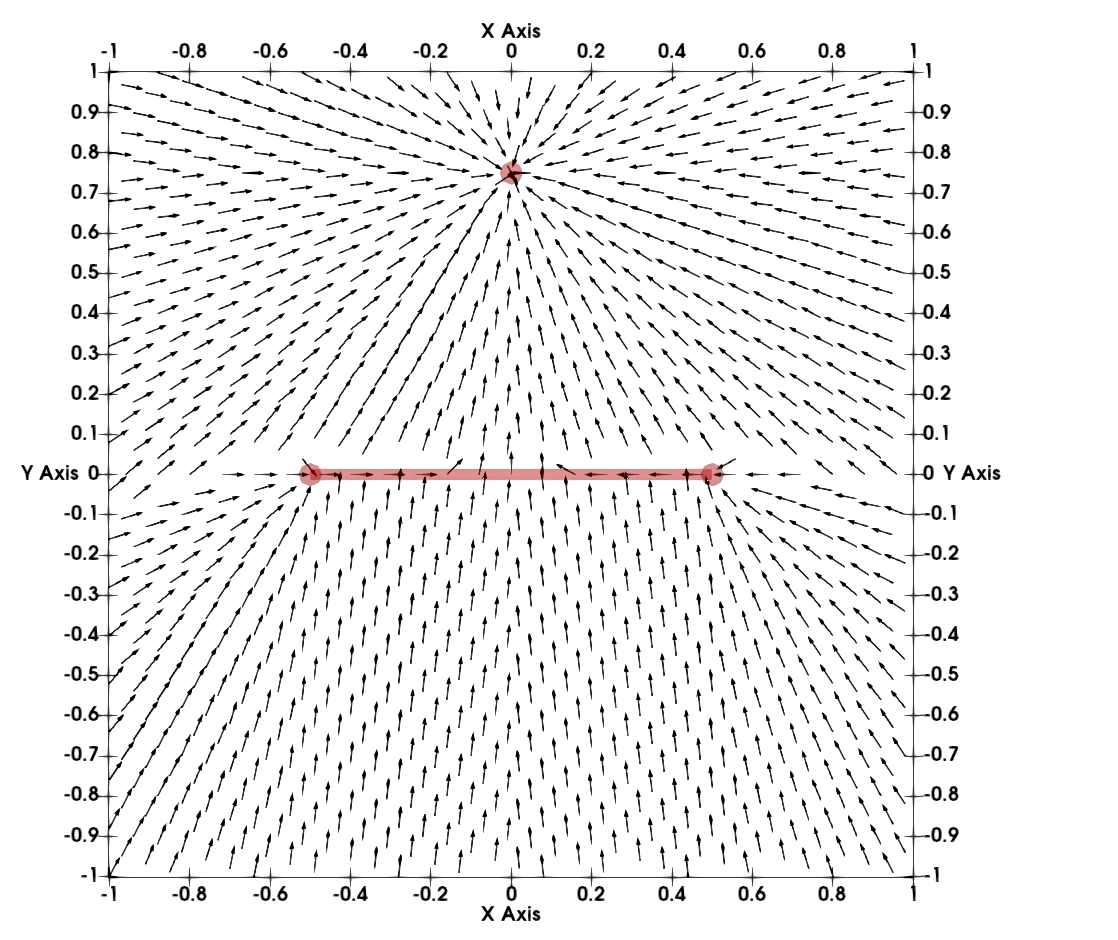} \\
	(a) & (b)\\[10pt]
	\includegraphics[width=0.49\textwidth]{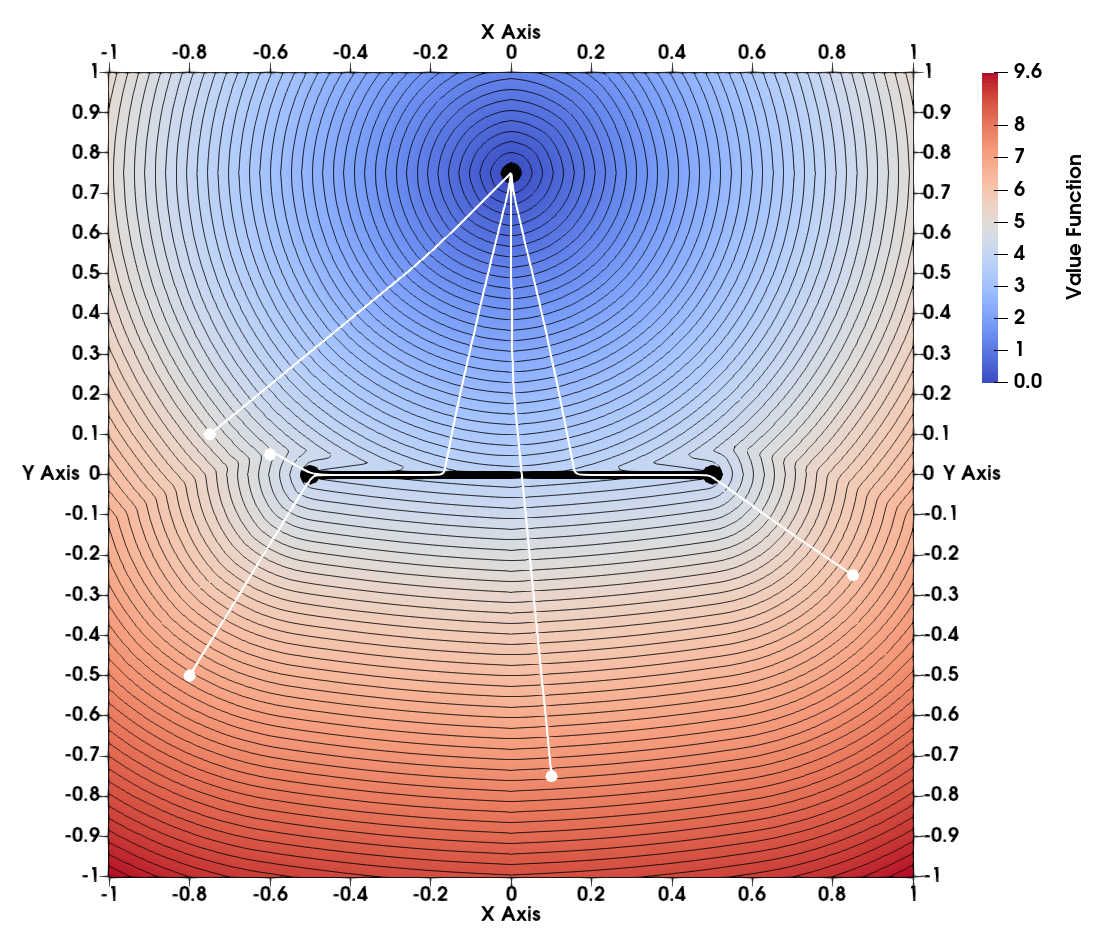} 
	&
	\includegraphics[width=0.49\textwidth]{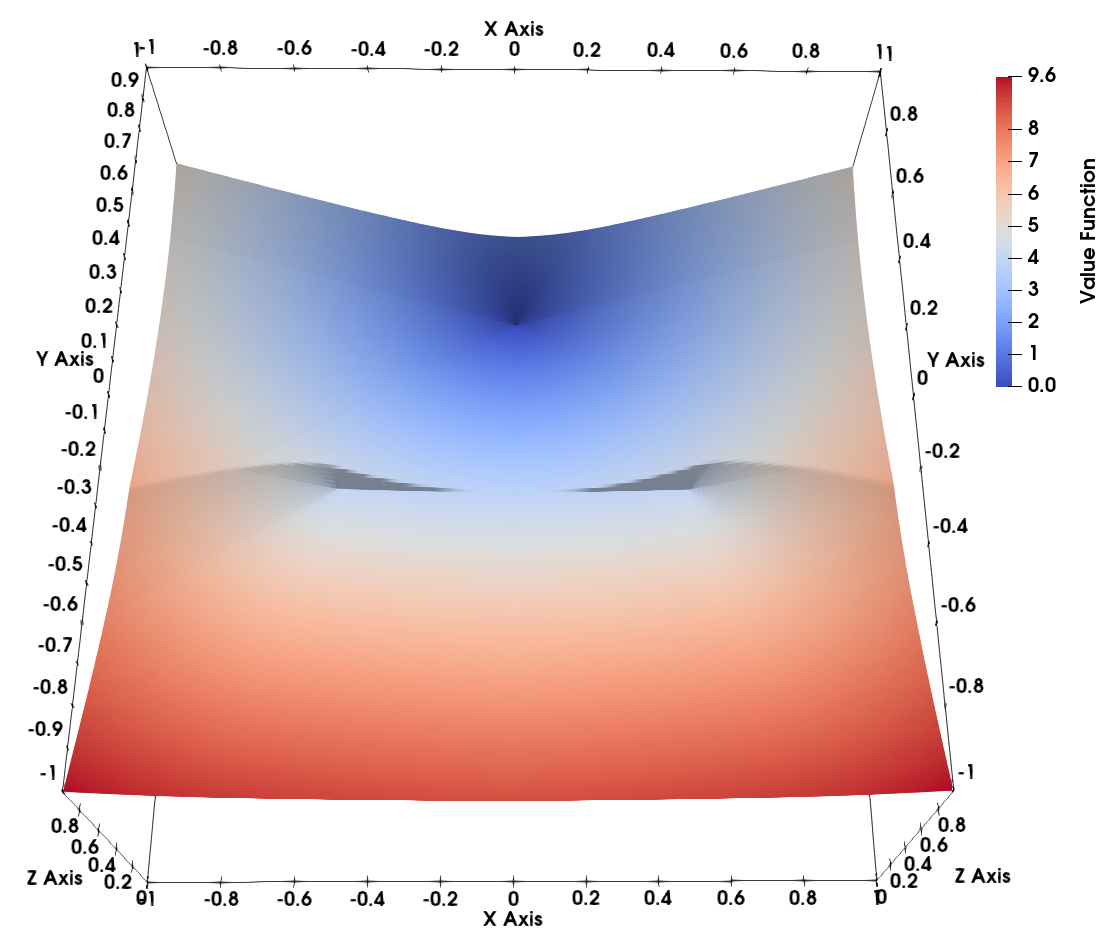} \\
(c) & (d)
\end{tabular}
\caption{Stratification for Test 1 (a), optimal dynamics (b), solution values and optimal trajectories (c), solution surface (d).}\label{Test1}}
\end{figure}

\noindent{\bf Test 2.}
We consider a more complex stratification of $\R^2$, obtained combining two vertical segments, their end-points and an additional isolated point. More precisely, we set $$\mathbf{M}^0=\{P_0\}\cup\{ P_{1,0}\}\cup\{ P_{1,1}\}\cup \{P_{2,0}\}\cup\{ P_{2,1}\}$$
with
$$P_0=\left(0,\frac34\right),\quad P_{1,0}=\left(-\frac12,-\frac12\right),\quad P_{1,1}=\left(-\frac12,\frac12\right),\quad P_{2,0}=\left(\frac12,-\frac12\right),\quad P_{2,1}=\left(\frac12,\frac12\right),$$
while
$$\mathbf{M}^1=L_0\cup L_1$$
with
$$L_0=\left\{-\frac12\right\}\times\left(-\frac12,\frac12\right),\quad L_1=\left\{\frac12\right\}\times\left(-\frac12,\frac12\right)$$
and 
$$\mathbf{M}^2=S:=\R^2\setminus(\overline{\mathbf{M}^0\cup\mathbf{M}^1})\,.$$
On each connected component of these submanifolds, we define the dynamics, the running cost and the discount factor according to the following table:
\begin{center}
\begin{tabular}{|c|c|c|c|}
\hline
     Connected component & Dynamics & Running Cost & Discount factor  \\
     \hline
\rule{0pt}{2.5ex}     $P_0$ & - & $\ell^{P_0}=0$ & $c^{P_0}=1$\\
      \hline
 \rule{0pt}{2.5ex}    $P_{i,j}\quad i=1,2\quad j=0,1$ & - & $\ell^{P_{i,j}}=1$ & $c^{P_{i,j}}=10^{-4}$\\
      \hline
 \rule{0pt}{2.5ex}$L_0$ & $b^{L_0}(x,y)\equiv 2$ & $\ell^{L_0}(x,y)\equiv 1$ & $c^{L_0}=10^{-4}$\\
 \hline
 \rule{0pt}{2.5ex} $L_1$ & $b^{L_1}(x,y)\equiv 3$ & $\ell^{L_1}(x,y)\equiv 1$ & $c^{L_1}=10^{-4}$\\
 \hline
\rule{0pt}{2.5ex}  $S$ & $b^{S}(x,y)\equiv 1$ & $\ell^{S}(x,y)\equiv 1$ & $c^{S}=10^{-4}$\\
 \hline
\end{tabular}
\end{center}
This stratification is reported in Figure \ref{Test2}-a, where now the color-map represents the values of the speed functions on the different submanifolds (while the running cost, excluding $\mathbf{M}^0$, is everywhere equal to one). 
Again, due to the values close to zero of all the discount factors (with the exception of $c^{P_0}$), this example resembles a minimum time problem with target $P_0$, but on the two vertical segments $L_0$ and $L_1$ we have speeds higher than the one on $S$. Hence, we can expect $L_0$ and $L_1$ to behave as fast tracks for the optimal trajectories.\\
This is confirmed by the numerical results reported in Figure \ref{Test2}-b (optimal dynamics) and in Figure \ref{Test2}-c (value function, level sets and optimal trajectories). It is worth noting that the solution is not symmetric with respect to the vertical axis, the higher speed on $L_1$ attracts more points than $L_0$. Moreover, we observe that the shock in the solution produced by $L_1$ ends up exactly at the end-point $P_{2,1}$, while the one produced by $L_0$ stops shortly before the end-point $P_{1,1}$ (see also Figure \ref{Test2}-d). This explains why all the optimal trajectories, running on $L_1$, leave the fast track at $P_{2,1}$, before pointing to $P_0$ along a straight line, while the ones running on $L_0$ leave it slightly before $P_{1,1}$. The result clearly depends on the balance between the different control problems involved in a neighborhood of $P_{1,1}$.\\ 
Finally, we observe that, also in this test, the contribution of the end-points of the two segments results irrelevant for global control problem on the stratification.   

\begin{figure}[!h]
\centering{
\begin{tabular}{cc}
	\includegraphics[width=0.49\textwidth]{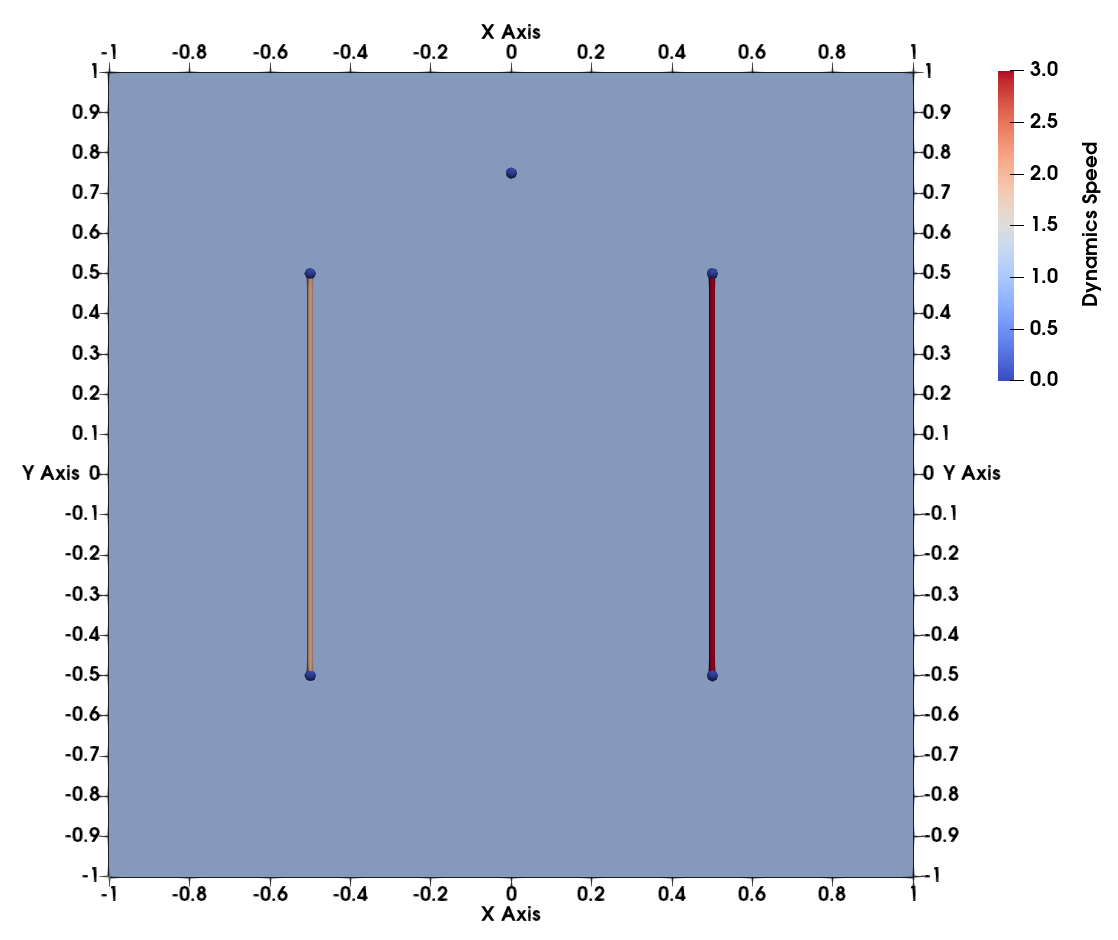} 
	&
	\includegraphics[width=0.49\textwidth]{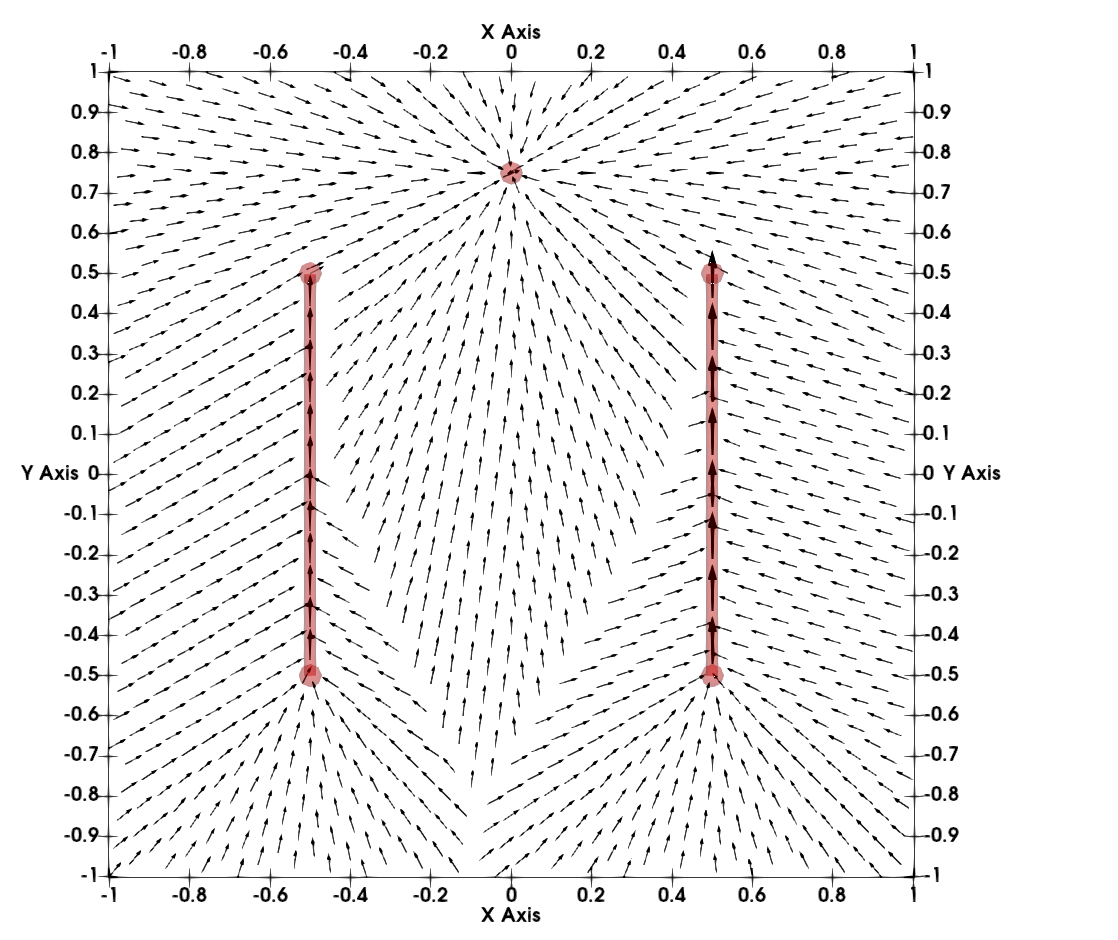} \\
	(a) & (b)\\[10pt]
	\includegraphics[width=0.49\textwidth]{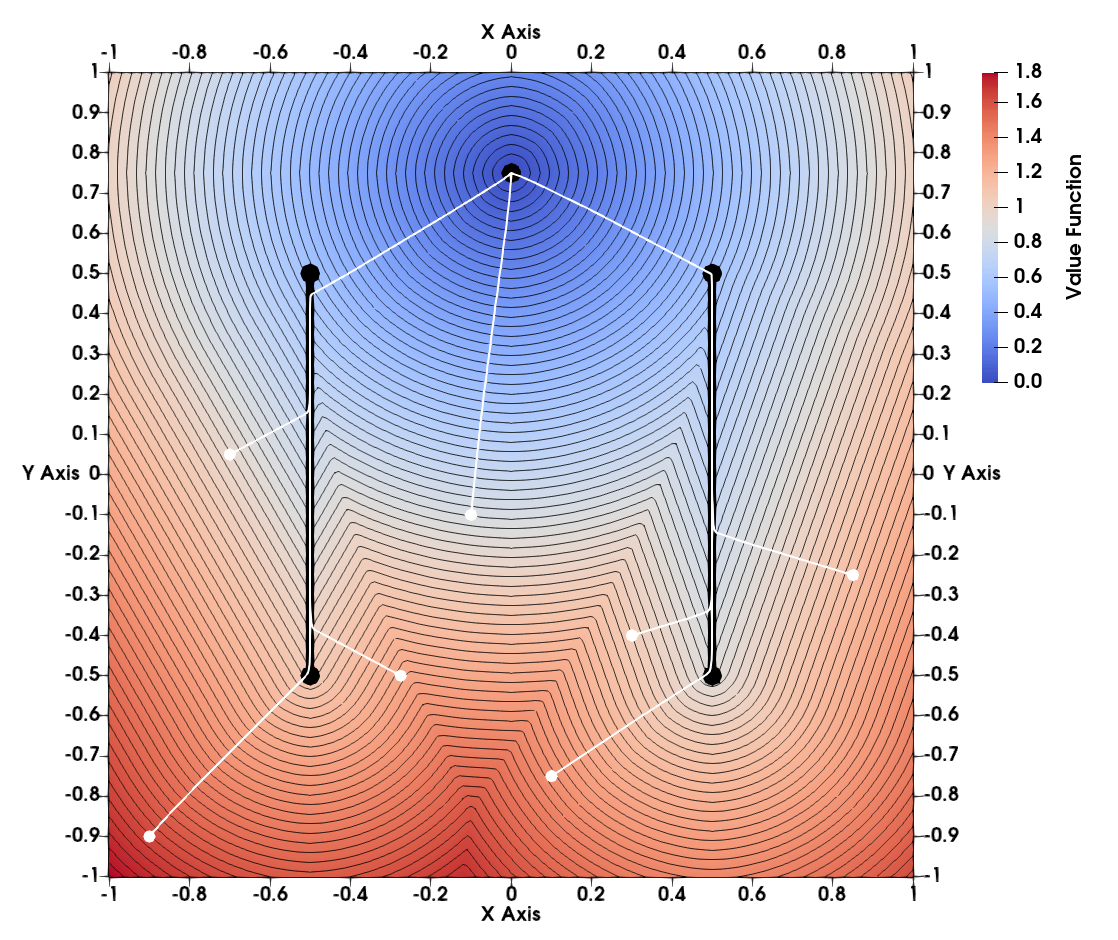} 
	&
	\includegraphics[width=0.49\textwidth]{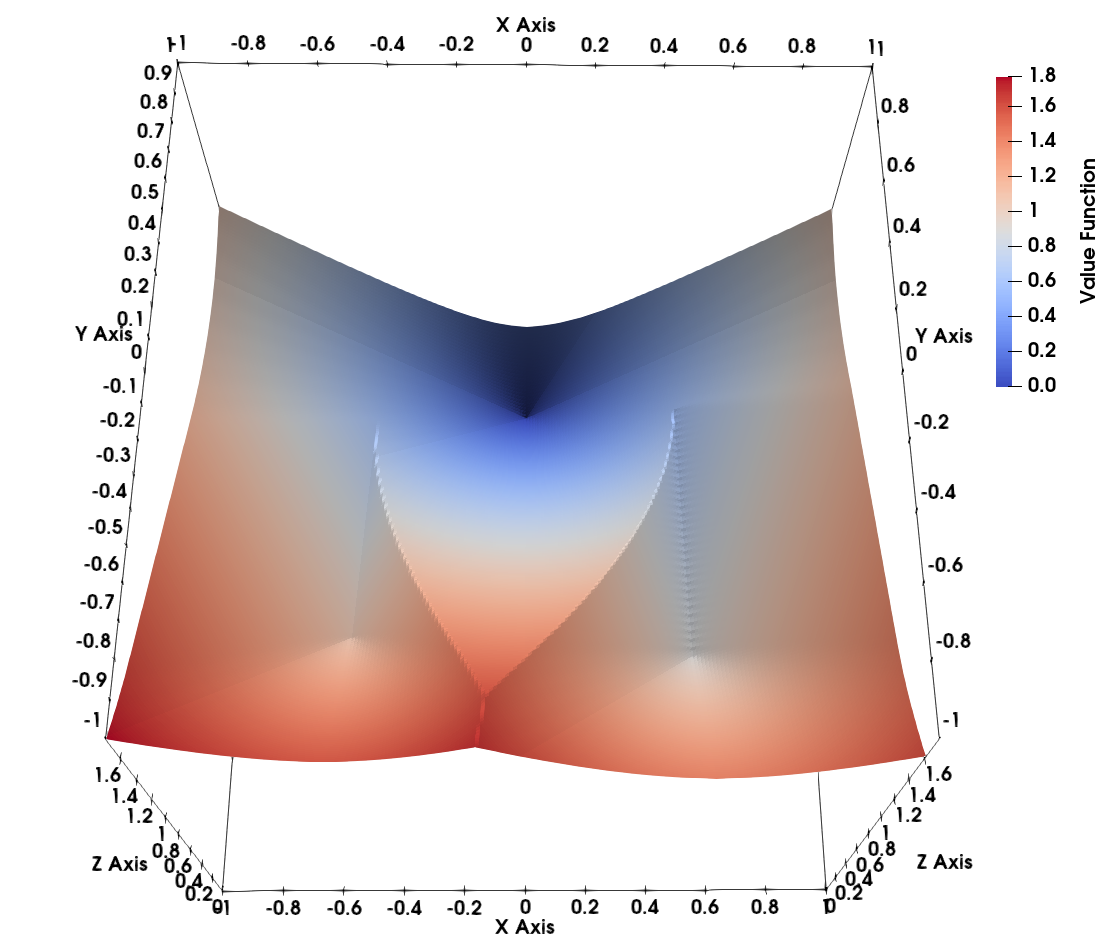} \\
(c) & (d)
\end{tabular}
\caption{Stratification for Test 2 (a), optimal dynamics (b), solution values and optimal trajectories (c), solution surface (d).}\label{Test2}}
\end{figure}

\noindent{\bf Test 3.}
We consider an even more complex stratification of $\R^2$, obtained combining the sides of a square and its corner points. Differently from the previous tests, here we also have two connected components for the two dimensional submanifold. More precisely, we set $$\mathbf{M}^0=\{P_0\}\cup\{ P_1\}\cup\{ P_2\}\cup \{P_3\}$$
with
$$P_0=\left(-\frac34,-\frac34\right),\quad P_1=\left(\frac34,-\frac34\right),\quad P_2=\left(\frac34,\frac34\right),\quad P_3=\left(-\frac34,\frac34\right),$$
while
$$\mathbf{M}^1=L_0\cup L_1\cup L_2 \cup L_3$$
with
\medskip
$$L_0=\left\{-\frac34\right\}\times\left(-\frac34,\frac34\right),\quad L_1=\left(-\frac34,\frac34\right)\times \left\{-\frac34\right\},$$
\medskip
$$L_2=\left\{\frac34\right\}\times\left(-\frac34,\frac34\right),\quad L_3=\left(-\frac34,\frac34\right)\times \left\{\frac34\right\}$$
and 
\medskip
$$\mathbf{M}^2=\R^2\setminus(\overline{\mathbf{M}^0\cup\mathbf{M}^1})=S_0\cup S_1$$
with
$$S_0=\left\{\|(x,y)\|_\infty<\frac34\right\},\quad S_1=\left\{\|(x,y)\|_\infty>\frac34\right\}.$$
On each connected component of these submanifolds, we define the dynamics, the running cost and the discount factor according to the following table:
\begin{center}
\begin{tabular}{|c|c|c|c|}
\hline
     Connected component & Dynamics & Running Cost & Discount factor  \\
     \hline
\rule{0pt}{2.5ex}     $P_{i}\quad i=0,1,2,3$ & - & $\ell^{P_i}=1$ & $c^{P_i}=10^{-4}$\\
 \hline
\rule{0pt}{2.5ex} $L_i\quad i=0,1,2,3$ & $b^{L_i}(x,y)\equiv 10$ & $\ell^{L_i}(x,y)\equiv 1$ & $c^{L_i}=10^{-4}$\\
 \hline
\rule{0pt}{2.75ex}  $S_0$ & $b^{S_0}(x,y)\equiv 1$ & $\ell^{S_0}(x,y)=\min\left\{ \cos(\frac83\pi x)+\cos(\frac83\pi y)+2,3\right\}$ & $c^{S_0}=1$
\\
 \hline
\rule{0pt}{2.5ex}  $S_1$ & $b^{S_1}(x,y)\equiv 1$ & $\ell^{S_1}(x,y)\equiv 1$ & $c^{S_1}=10^{-4}$\\
 \hline
\end{tabular}
\end{center}
\medskip
This stratification is reported in Figure \ref{Test3}-a, with the color-map representing the values of the running costs on the different submanifolds.\\
Starting from $S_1$,  the running cost jumps to higher values across the four segments, where we recognize eight local minimizers, two for each segment. Despite this barrier, the running cost also has four absolute minimizers inside $S_0$ which act again as target points, due to the very small discount factors on $S_1$ and $\mathbf{M}^0\cup\mathbf{M}^1$.\\
The numerical results are reported in Figure \ref{Test3}-b (optimal dynamics) and in Figure \ref{Test3}-c (value function, level sets and optimal trajectories). As expected, we observe that all the optimal trajectories eventually reach the target points, but the speed on $\mathbf{M}^1$ is so high that a longer travelled distance is always preferred, in order to approach $S_0$ at the local minimizers for the running cost. This effect also propagates inside $S_0$, in order to avoid the absolute maxima (plateaux) of $\ell^{S_0}$ located around the four corners and the middle points of the segments. The result is the creation of several shocks with triple points, see Figure \ref{Test3}-d).  
Finally, the contribution of the four corners to the optimal solution is again irrelevant.   
\begin{figure}[!h]
\centering{
\begin{tabular}{cc}
	\includegraphics[width=0.49\textwidth]{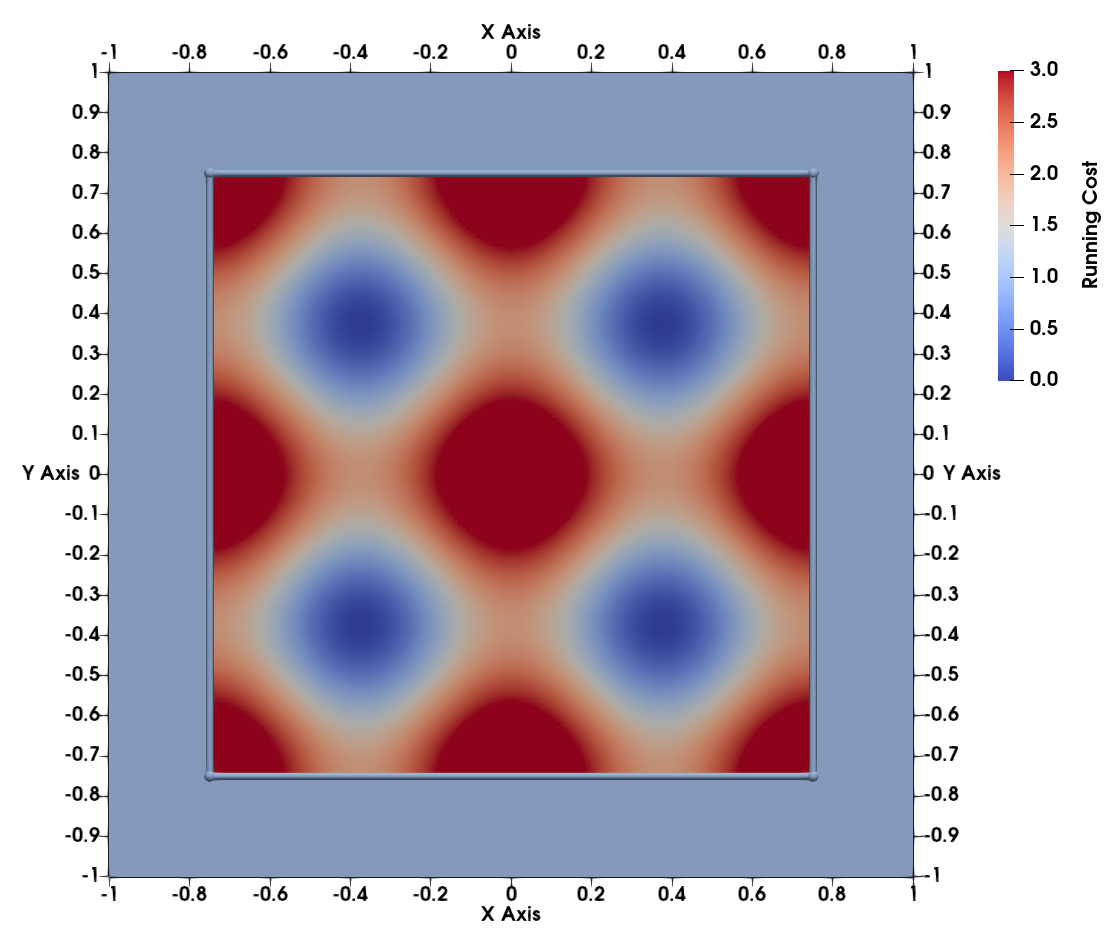} 
	&
	\includegraphics[width=0.49\textwidth]{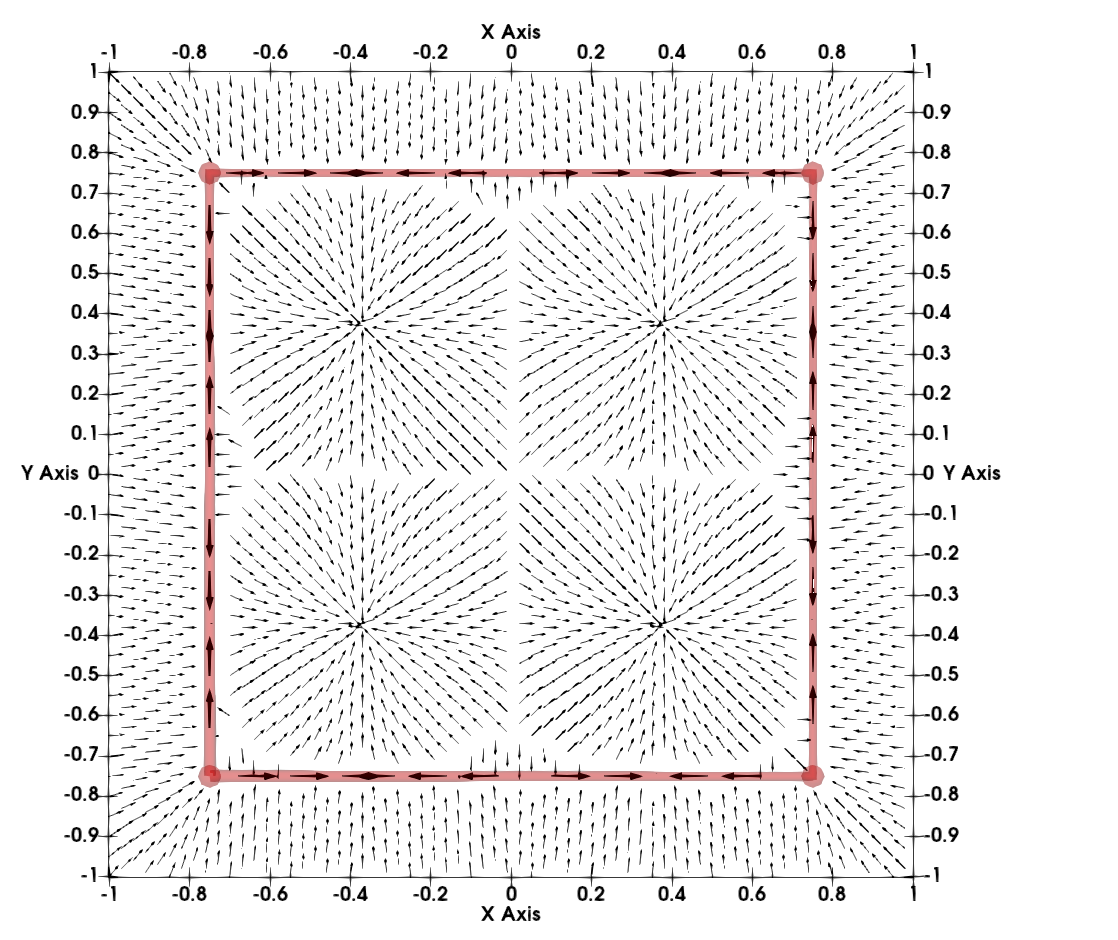} \\
	(a) & (b)\\[10pt]
	\includegraphics[width=0.49\textwidth]{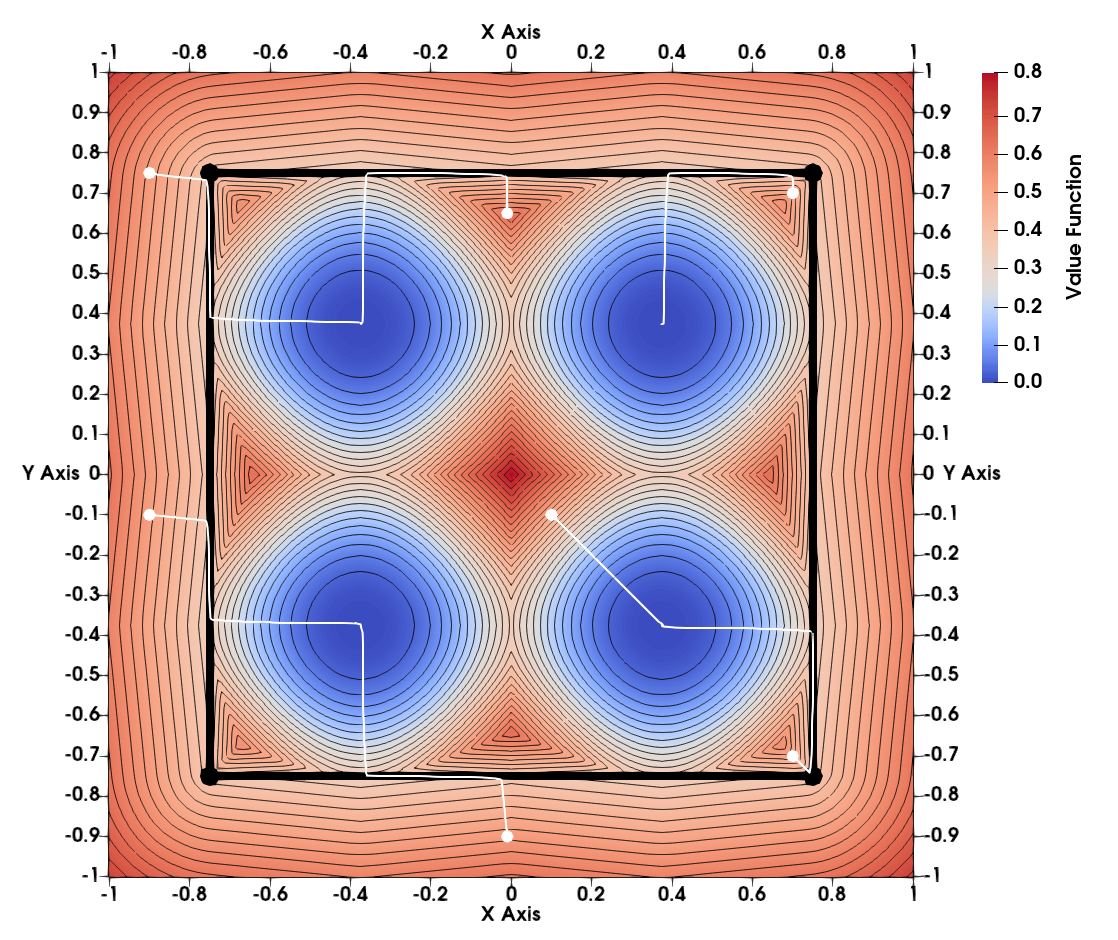} 
	&
	\includegraphics[width=0.49\textwidth]{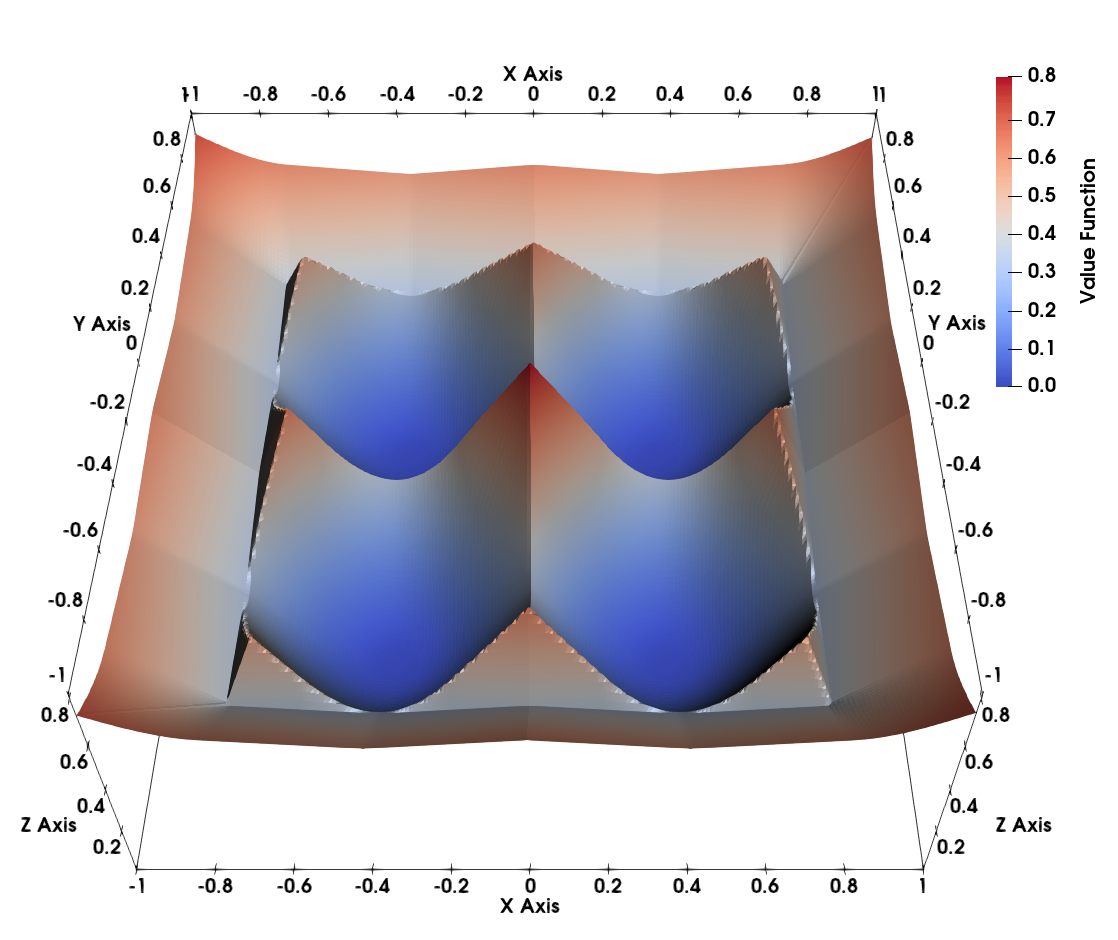} \\
(c) & (d)
\end{tabular}
\caption{Stratification for Test 3 (a), optimal dynamics (b), solution values and optimal trajectories (c), solution surface (d).}\label{Test3}}
\end{figure}\\

\noindent{\bf Test 4.} This last experiment is similar to Test 2, but in three dimensions. We consider a stratification of $\R^3$ obtained combining a horizontal square (including its sides and corners) with a vertical segment (including its end-points). 

\noindent More precisely, we set
 $$\mathbf{M}^0=\{P_{0,0}\}\cup\{ P_{0,1}\}\cup\{ P_{1,0}\}\cup \{P_{1,1}\}\cup\{ P_{1,2}\}\cup\{ P_{1,3}\}$$
with
$$P_{0,0}=\left(0,0,\frac12\right),\quad P_{0,1}=\left(0,0,0\right),$$
\medskip
$$
P_{1,0}=\left(-\frac12,-\frac12,0\right),\quad P_{1,1}=\left(\frac12,-\frac12,0\right),\quad P_{1,2}=\left(\frac12,\frac12,0\right),\quad P_{1,3}=\left(-\frac12,\frac12,0\right),$$
while
$$\mathbf{M}^1=L_0\cup L_{1,0}\cup L_{1,1}\cup L_{1,2}\cup L_{1,3}$$
with
$$L_0=\left\{0\right\}\times\left\{0\right\}\times \left(0,\frac12\right),$$
\medskip
$$L_{1,0}=\left\{-\frac12\right\}\times\left(-\frac12,\frac12\right)\times\left\{0\right\},\quad L_{1,1}=\left(-\frac12,\frac12\right)\times \left\{-\frac12\right\}\times\left\{0\right\},$$
\medskip
$$L_{1,2}=\left\{\frac12\right\}\times\left(-\frac12,\frac12\right)\times\left\{0\right\},\quad L_{1,3}=\left(-\frac12,\frac12\right)\times \left\{\frac12\right\}\times\left\{0\right\},$$
and
$$\mathbf{M}^2=S:=\left\{\|(x,y,0)\|_\infty<\frac12\right\}\times\left\{0\right\},\quad\mathbf{M}^3=V:=R^3\setminus(\overline{\mathbf{M}^0\cup\mathbf{M}^1\cup\mathbf{M}^2}).$$
On each connected component of these submanifolds, we define the dynamics, the running cost and the discount factor according to the following table:
\begin{center}
\begin{tabular}{|c|c|c|c|}
\hline
     Connected component & Dynamics & Running Cost & Discount factor  \\
     \hline
\rule{0pt}{2.5ex} $P_{0,0}$ & - & $\ell^{P_{0,0}}=0$ & $c^{P_{0,0}}=1$\\
     \hline
\rule{0pt}{2.5ex} $P_{0,1}$ & - & $\ell^{P_{0,1}}=1$ & $c^{P_0}=10^{-4}$\\
 \hline
\rule{0pt}{2.5ex} $P_{1,i}$\quad $i=0,1,2,3$ & - & $\ell^{P_{1,i}}=1$ & $c^{P_{1,i}}=10^{-4}$\\
      \hline
\rule{0pt}{2.5ex} $L_0$ & $b^{L_0}(x,y,z)\equiv 5$ & $\ell^{L_0}(x,y,z)\equiv 1$ & $c^{L_0}=10^{-4}$\\
 \hline
 \rule{0pt}{2.5ex} $L_{1,i}\quad i=0,1,2,3$ & $b^{L_{1,i}}(x,y,z)\equiv 5$ & $\ell^{L_{1,i}}(x,y,z)\equiv 1$ & $c^{L_{1,i}}=10^{-4}$\\
 \hline
 \rule{0pt}{2.5ex} $S$ & $b^{S}(x,y,z)\equiv 5$ & $\ell^{S}(x,y,z)\equiv 1$ & $c^{S}=10^{-4}$\\
 \hline
\rule{0pt}{2.5ex}  $V$ & $b^{V}(x,y,z)\equiv 1$ & $\ell^{V}(x,y,z)\equiv 1$ & $c^{V}=10^{-4}$\\
 \hline
\end{tabular}
\end{center}
This stratification is reported in Figure \ref{Test4}-a, where the color-map represents the values of the speed functions on the different submanifolds.\\
According to the choice of the discount factors, only $P_{0,0}$ acts as a target point, but on  $\mathbf{M}^1\cup\mathbf{M}^2$ we have a speed for the dynamics higher than the one on $V$. As in Test 2, we expect that the optimal trajectories starting from points far enough from $P_{0,0}$ will be attracted by these submanifolds, increasing their travelled distances before approaching the target.\\
The numerical results are reported in Figure \ref{Test4}-b (some iso-surface of the value function and some level-sets on a couple of two dimensional slices) and in Figure \ref{Test4}-c (optimal trajectories).\\
Looking at the level-sets on the horizontal slice, we observe that the solution behaves like the distance function from the boundary of $S$. On the other hand, entering $S$, the solution has circular level sets, behaving like the distance function from the point $P_{0,1}$. In particular, it is interesting to note that, once a trajectory starting from $V$ approaches $S$, it never leaves $S$, but proceeds along a straight line towards $P_{0,1}$, then it goes up along the segment $L_0$ towards the target.\\
A closer look at the iso-surfaces of the solution also reveals the shock surface resulting from the competition between $S$ and $L_0$.\\
Finally, we remark that, as in the previous tests, the contribution to the optimal solution of the four segments $L_{1,i}$ for $i=0,\dots,3$, of the four points $P_{1,i}$ for $i=0,\dots,3$, and of the end-point $P_{0,1}$ of $L_0$, is again irrelevant.   
\begin{figure}[!h]
\centering{
\begin{tabular}{c}
\begin{tabular}{cc}
	\includegraphics[width=0.49\textwidth]{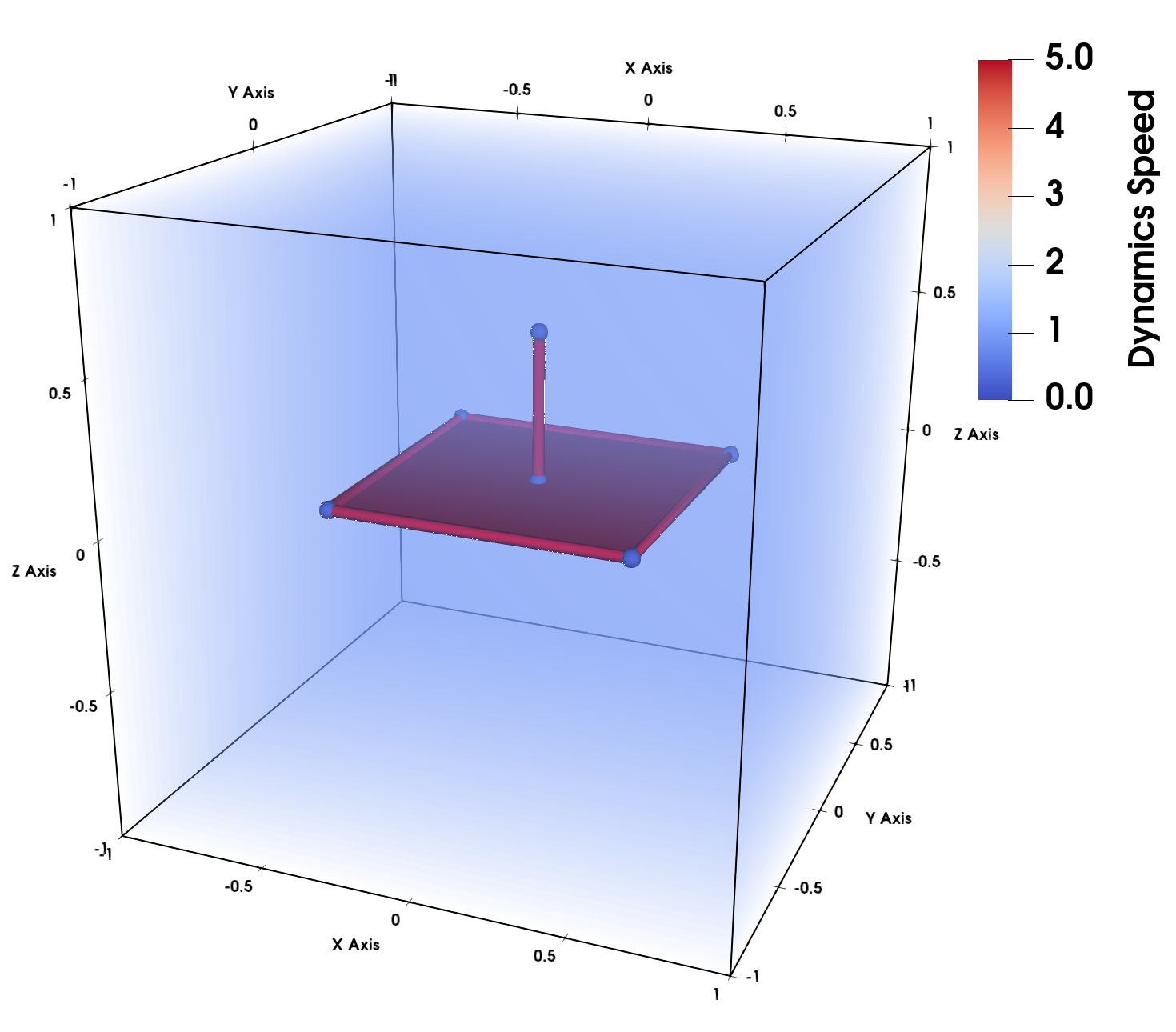} 
	&
	\includegraphics[width=0.49\textwidth]{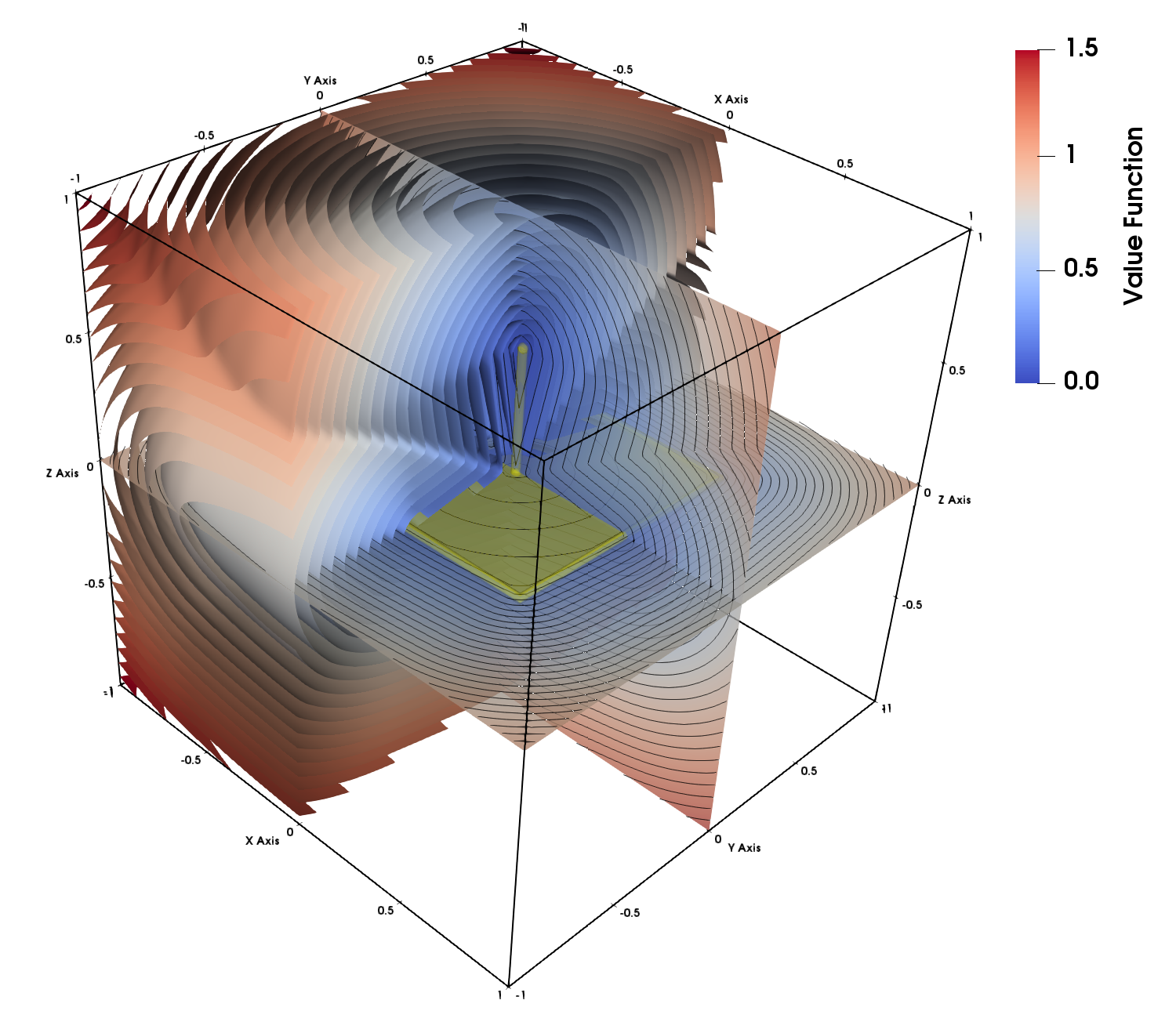} \\
	(a) & (b)
\end{tabular}
\\
	\includegraphics[width=0.65\textwidth]{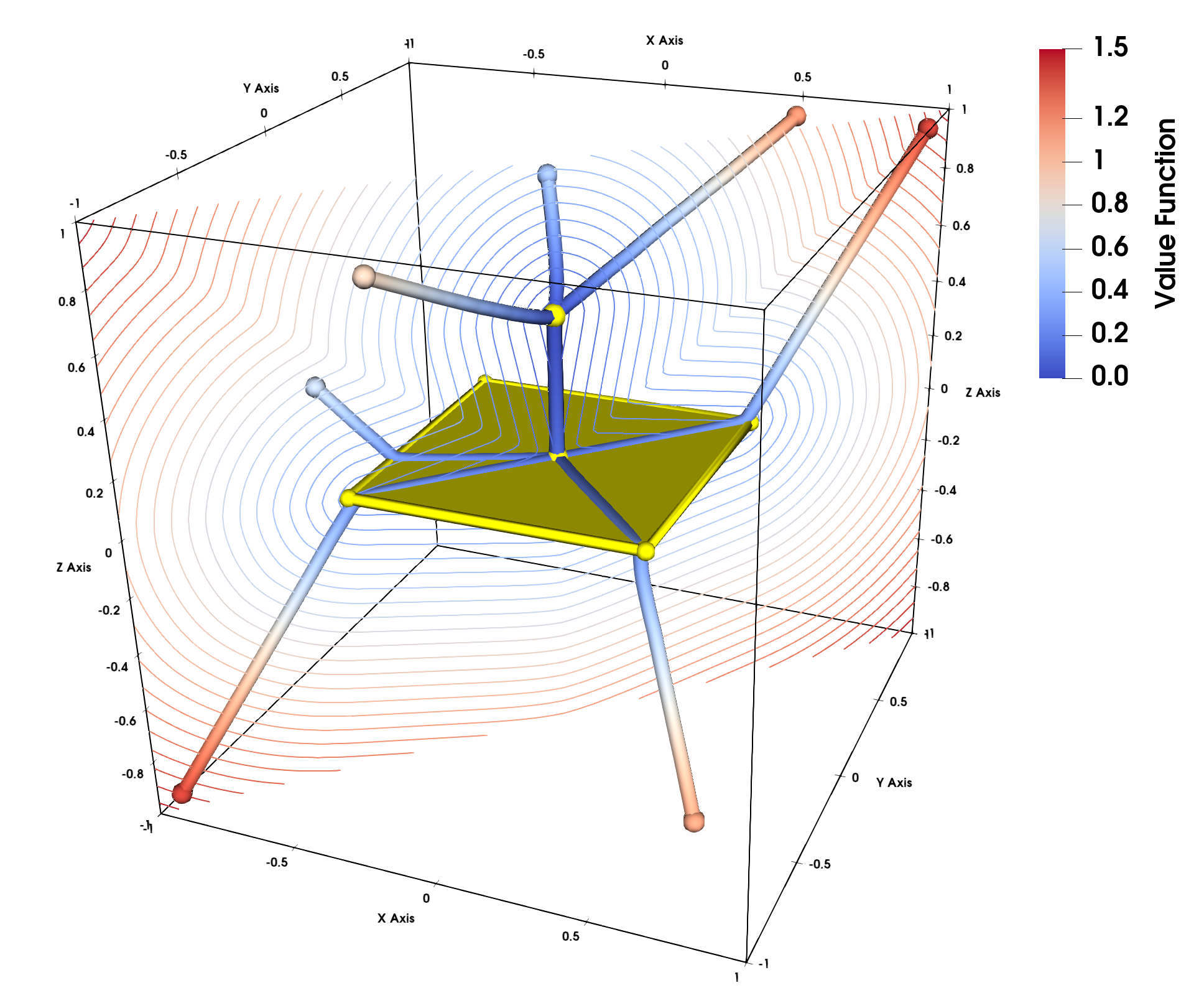} 
 \\
(c)
\end{tabular}
\caption{Stratification for Test 4 (a), solution surfaces and values (b), optimal trajectories (c).}\label{Test4}}
\end{figure}

\paragraph*{Acknowledgments.} 
The authors would like to thank Guy Barles and Emmanuel Chasseigne for providing them with an updated version of the book \cite{BaCh_book}.
\bibliographystyle{plain}

\end{document}